\documentclass{article}

\usepackage{graphicx}
\usepackage{caption}
\usepackage{subcaption}
\usepackage{setspace}
\usepackage{hyperref}
\usepackage{amsmath}
\usepackage{amssymb}
\usepackage{amsthm}
\usepackage{mathtools}
\usepackage[a4paper, total={5.8in, 9.55in}]{geometry}
\usepackage{xcolor}

\newcommand{\bize}{\begin{itemize}}
\newcommand{\eize}{\end{itemize}}
\newcommand{\beq}{\begin{equation}}
\newcommand{\eeq}{\end{equation}}
\newcommand{\bluea}{(a)}
\newcommand{\blueb}{(b)}
\newcommand{\bluec}{(c)}
\newcommand{\blued}{(d)}
\newcommand{\bluee}{(e)}
\newcommand{\bluef}{(f)}

\newtheorem{definition}{\bf Definition}[section]

\newtheorem{remark}{\bf Remark}[section]
\newtheorem{proposition}{\bf Proposition}[section]

\title{
Phase tipping: How cyclic ecosystems respond to contemporary climate
}
\author{Hassan Alkhayuon\footnote{University College Cork, School of Mathematical Sciences, Western Road, Cork, T12 XF62, Ireland}, Rebecca C. Tyson\footnote{Department of Mathematics and Statistics, University of British Columbia
Okanagan, Kelowna, BC, Canada}, and Sebastian Wieczorek$^*$}

\date{June 2021}

\begin{document}
\maketitle
\begin{abstract}

We identify the {\em phase of a cycle} as a new critical factor for 
tipping points (critical transitions) in cyclic systems subject 
to time-varying external conditions. As an example, we consider 
how contemporary climate variability induces tipping from a 
predator-prey cycle to extinction in two paradigmatic predator-prey 
models with an Allee effect. 
Our analysis of these examples uncovers a counter-intuitive 
behaviour, which we call {  phase tipping} or {\em P-tipping},
where tipping to extinction occurs only from certain phases of the cycle. To explain this behaviour, we combine global dynamics with set theory  and introduce the concept of {\em partial basin instability} for 
{  attracting} limit cycles. This concept  provides a general framework to analyse and identify 
{  easily testable}
criteria for the occurrence of 
 {  phase tipping} in externally forced systems, {  and can be extended to more complicated attractors.}

\end{abstract}

\maketitle
\section{Introduction}

Tipping points or critical transitions are fascinating nonlinear phenomena that are known to occur in complex systems subject to changing external conditions or external inputs. They are ubiquitous in nature and, in layman's terms, can be described as large, sudden, and unexpected changes in the state of the system triggered by small or slow changes in the external inputs~\cite{scheffer2009,lenton2008}. Owing to potentially catastrophic and irreversible changes associated with tipping points, it is important to identify and understand the underlying dynamical mechanisms that enable such transitions.  
 To do so, it is helpful to consider base states (attractors for fixed external conditions) whose position or stability change as the external conditions vary over time.
Recent work on tipping from base states that are stationary (attracting equilibria) has been shown to result from three generic tipping mechanisms~\cite{ashwin2012}:

\begin{itemize}
    \item 
    {\em Bifurcation-induced tipping or B-tipping} occurs when the external input passes through a {dangerous bifurcation} of the base state, at which point
    {\em the base state disappears or turns unstable}, forcing the system to move to a different state~\cite{thompson1994,thompson2011,kuehn2011}.
     \item
     {\em Rate-induced tipping or R-tipping} occurs 
     when  {\em the external input  varies too fast}, so the system deviates too far from the moving base state and crosses some tipping threshold~\cite{wieczorek2011,perryman2014,vanselow2019,wieczorek2020,kuehn2020}, e.g. into the domain of attraction of a different state~\cite{scheffer2008,o2019,ashwin2017,alkhayuon2019,Hartl2019,kiers2020}. 
      The special case of delta-kick external input is referred to as {\em shock-tipping} or {\em S-tipping}~\cite{halekotte2020}. 
      In contrast to B-tipping, 
     R-tipping need not involve any  bifurcations of the base state.
     \item
     {\em Noise-induced tipping or N-tipping} occurs when external {\em random fluctuations}  drive the system too far from the base state and past some tipping threshold~\cite{franovic2018}, e.g. into the domain of attraction of a different state~\cite{benzi1981,ditlevsen2010,ritchie2016,chen2019}.
\end{itemize}

Many complex systems have non-stationary base states, meaning that these systems exhibit regular or irregular self-sustained oscillations for fixed external inputs~\cite{freund2006,medeiros2017,alkhayuon2018, bathiany2018,kaszas2019,keane2020,longo2020,lohmann2021}. 
Such base states open the possibility for other generic tipping mechanisms when the external inputs vary over time.
In this paper, we focus on tipping from the next most complicated base state, a periodic state (attracting limit cycle), and identify a new tipping mechanism: 
\begin{itemize}
    \item {  Phase tipping (Partial tipping {\cite{alkhayuon2018}}) or P-tipping} occurs when a {\em too fast change} or {\em random fluctuations}
    in the external input cause the system to tip to a different state, but only from {\em certain  phases} { \em (or certain parts)} of the base state
    {  and its neighbourhood}.
    In other words, the system has to be in the right phases to tip, whereas no tipping occurs from other phases. 
\end{itemize}

The concept of P-tipping naturally extends to more complicated 
quasiperiodic (attracting tori) and  chaotic (strange attractors) 
base states and, in a certain sense, unifies the notions of R-tipping, S-tipping and
N-tipping. A simple intuitive picture is
that external inputs can trigger the system past some tipping
threshold, but only from  {\em certain parts} of the base state {  and its neighbourhood}. Thus, 
P-tipping can also be interpreted as {\em partial tipping}. Indeed, 
examples of P-tipping with smoothly changing external inputs include the recently studied
``partial R-tipping" from periodic base states~\cite{alkhayuon2018}, 
and probabilistic tipping  from chaotic base states~\cite{ashwin2021,kaszas2019,lohmann2021}. 
Furthermore, P-tipping offers 
new insight into classical phenomena such as stochastic resonance~\cite{benzi1981,anishchenko1993,gammaitoni1998}, where noise-induced transitions 
between coexisting non-stationary states occur (predominantly) from 
certain phases of these states and at an optimal noise strength. 
Other examples of P-tipping due to random fluctuations include ``state-dependent vulnerability of synchronization" in complex networks~\cite{medeiros2019}, and ``phase-sensitive excitability" from periodic states~\cite{franovic2018}, which can be interpreted as partial N-tipping.

 Here, we construct a general mathematical framework to analyse {\em irreversible P-tipping from periodic base states}. By ``irreversible" we mean that the system 
 approaches a different state in the long-term.  The framework allows us to explain counter-intuitive properties, identify the underlying dynamical mechanism, and give easily testable criteria for the occurrence of P-tipping.
Furthermore, motivated by growing evidence that tipping points in the Earth system 
could be more likely than was thought~\cite{lenton2008, rockstrom:2009, barnosky:2012}, we show that P-tipping could occur in real ecosystems subject to contemporary climate change. To be more specific, we uncover robust P-tipping from {  predator-prey oscillations} 
to extinction due to climate-induced decline in  prey resources in two paradigmatic predator-prey models with an Allee effect: the Rosenzweig-McArthur (RMA) model~\cite{rosenzweig1963} and the May (or Leslie-Gower-May) model~\cite{May2019}. 
{  Intuitively, the phase sensitivity of tipping from predator-prey  oscillations 
arises because a given drop in prey resources has distinctively different effects when applied during the phases of the cycle with the fastest growth and the fastest decline of  prey.}
Both the RMA and May models have been used to study predator-prey interactions in a 
number of natural systems~\cite{vitense:2016, sauve:2020, sarker:2020}. 
Here, we use realistic parameter values for the Canada lynx and snowshoe hare system~\cite{tyson:2009, strohm:2009}, together with real climate records from various communities in the boreal and deciduous-boreal forest~\cite{Marley2020}. 

The nature of predator-prey interactions often leads to regular, high amplitude, multi-annual cycles~ \cite{kendall1998}. Consumer-resource and host-parasitoid interactions are similar, and also often lead to dramatic cycles~\cite{turchin:2003}.  In insects, cyclic outbreaks can be a matter of deep economic concern, as the sudden increase in defoliating insects leads to significant crop damage~\cite{stewart:1993}.  In the boreal forest, one of the most famous predator-prey cycles is that of the Canada lynx and snowshoe hare~\cite{turchin:2003}.  The Canada lynx is endangered in parts of its southern range, and the snowshoe hare is a keystone species in the north, relied upon by almost all of the mammalian and avian predators there~\cite{krebs:2001}.  These examples illustrate the ubiquitous nature of cyclic predator-prey interactions, and their significant economic and environmental importance.  Their persistence in the presence of climate change is thus a pressing issue.  

Anthropogenic and environmental factors are subjecting cyclic predator-prey systems to external forcing which, through climate change, is being altered dramatically in both spatial and time-dependent ways~\cite{sauve:2020, vanderbolt:2018, richardson:2005, harley:2011, ummenhofer:2016, pauli:2003}.  In addition to long-term changes due to global warming, 
there is a growing interest in changes in climate variability on year-to-decade timescales, owing to its more imminent impacts~\cite{zimmerman:2020}.  In particular, increased variability of short-term climatic events manifests itself as, for example, larger hurricanes, hotter heatwaves, and more severe floods~\cite{elkenawy:2020, spinoni:2018, smale:2019, veh:2020, bloeschl:2019, depietri:2018, ummenhofer:2016, knutson:2020, chauvin:2020}. It is unknown how cyclic predator-prey systems will interact with these changes in climate variability.

Beyond ecology, oscillatory predator-prey interactions play an important role in finance and economics~\cite{mehlum2003, mehlum2006}. Thus, our work may also be relevant for understanding economies in developing countries~\cite{mirza2019}. Such economies are non-stationary by nature, and it  may well be that developing countries have only short phases in their development, or narrow windows of opportunity, during which  external investments can induce transitions from poverty to wealth.

This paper is organised as follows. In Section~\ref{sec:ppm}, we 
introduce the Rosenzweig-MacArthur and May models, 
define phase for the predator-prey {  oscillations}, and describe the 
random processes used  to model climatic variability. 
In Section~\ref{sec:cipst}, Monte Carlo simulations of the predator-prey models reveal counter-intuitive 
properties of P-tipping and highlight the key differences 
from B-tipping. 
In Section~\ref{sec:basin_instability}, we present a geometric framework for P-tipping and define the concept of 
{\em partial basin instability} for attracting limit cycles.
In Section~\ref{sec:basin_instability_PP}, we produce two-parameter bifurcation diagrams for the autonomous predator-prey frozen systems 
with fixed-in-time external inputs, identify bistability between predator-prey cycles and extinction, and  reveal parameter regions of partial basin instability - these cannot be captured by classical bifurcation analysis but are essential for understanding P-tipping.
Finally, we show that partial basin instability explains and gives testable 
criteria for the occurrence of P-tipping. We summarise our results in Section~\ref{sec:conclusions}.

\section{Oscillatory predator-prey models with varying climate.}
\label{sec:ppm}

{  We carry out our study of P-tipping in the context of two paradigmatic predator-prey models, which we present here.  We also define  "phase" in the context of the predator-prey limit cycles and nearby oscillations. Finally, we introduce our climate variability model.}

\subsection{The Rosenzweig-MacArthur and May models}

\begin{figure}[]
    \centering 
    \includegraphics[width=1\textwidth]{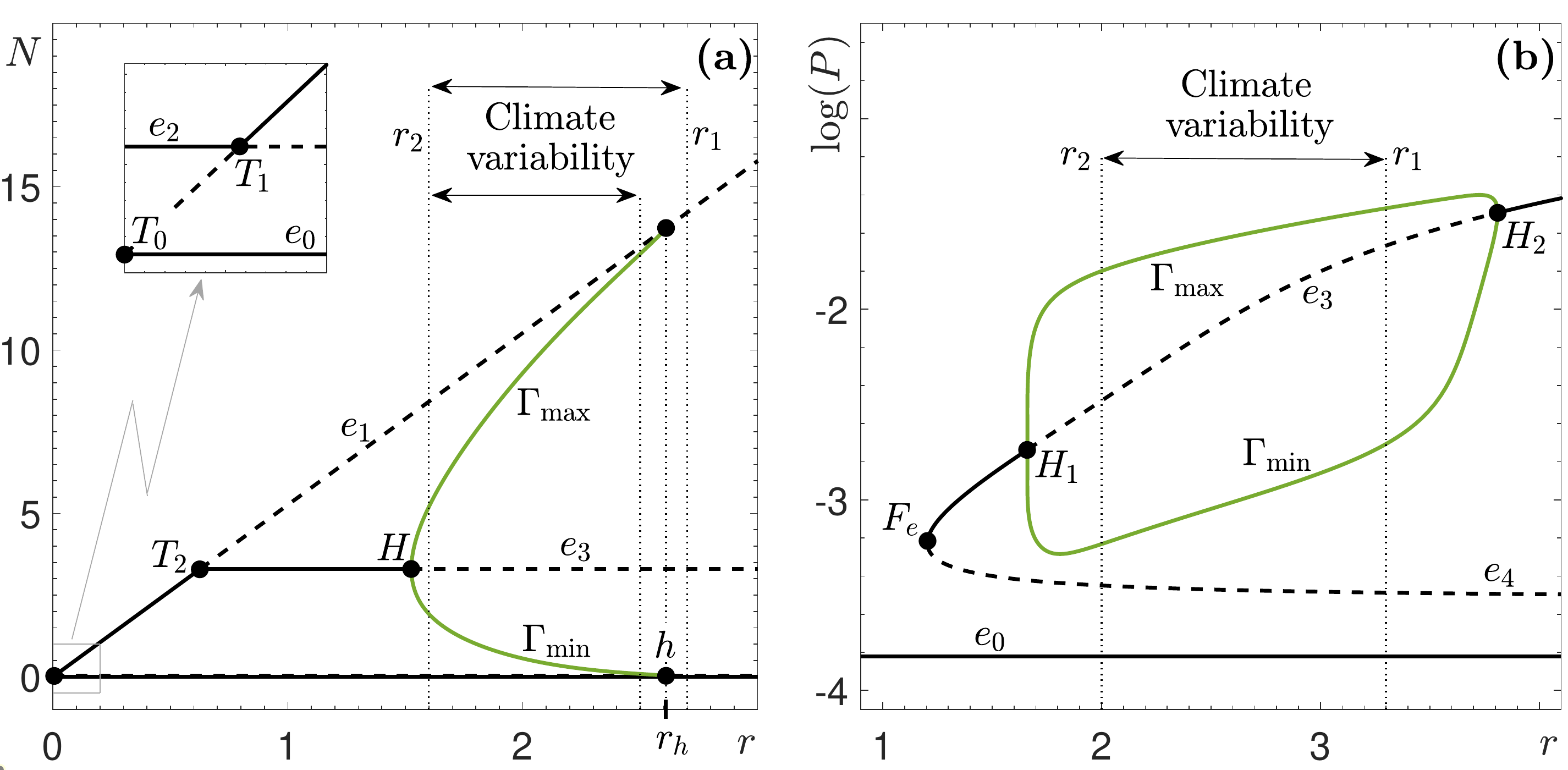}
    \caption{One-parameter bifurcation diagrams 
    with different but fixed-in-time $r$ for (a) the autonomous RMA frozen model~\eqref{eq:RM1} and (b) the autonomous May frozen model~\eqref{eq:May1}. The other parameter values are given in Appendix~\ref{app:maymodel}, Table~\ref{tab:parValues}. 
    }
    \label{fig:1parBD}
\end{figure}

The Rosenzweig-MacArthur (RMA) model~\cite{rosenzweig1963,vanselow2019} describes the time evolution of interacting prey $N$ and predator $P$ populations~\cite{Berryman1992}:
\begin{equation}\label{eq:RM1}
\begin{split}
    \dot{N}&=   r(t)\,N\left( 1- \frac{c}{r(t)} N\right)\left(\frac{N-\mu}{\nu+N}\right) - \frac{\alpha NP}{\beta + N},\\
    \dot{P} &= \chi\,\frac{\alpha NP}{\beta +N} - \delta P.
    \end{split}
\end{equation}
In the prey equation, $-r(t)\mu/\nu$ is the low-density (negative) prey growth rate, $c\mu/\nu$  quantifies the nonlinear modification of the low-density
prey growth,
the term $(N-\mu)/(\nu + N)$ gives rise to the strong Allee effect that accounts for negative prey growth rate at low prey population density,
$\alpha$ is the saturation predator kill rate, and $\beta$ is the predator kill half-saturation constant. 
The ratio $r(t)/c$ is often referred to as the {\em carrying capacity} of the ecosystem. It is the maximum prey population that can be sustained by the environment in the absence of predators~\cite{strohm:2009}.
In the predator equation, $\chi$ represents the prey-to-predator conversion ratio, and $\delta$ is the predator mortality rate. 
Realistic parameter values, estimated from Canada lynx and snowshoe hare data \cite{strohm:2009,tyson:2009}, can be found in Appendix \ref{app:maymodel}, Table~\ref{tab:parValues}.

As we explain in Sec.~\ref{sec:ppm}\ref{sec:climvar}, $r(t)$ is a piecewise constant function of time that describes the varying climate.
This choice makes the nonautonomous system~\eqref{eq:RM1} piecewise autonomous in the sense that it behaves like an autonomous system over finite time intervals. Therefore, much can be understood about the behaviour of the nonautonomous system~\eqref{eq:RM1}  by looking at the autonomous {\em frozen system} with different but fixed-in-time values of $r$. 

The RMA frozen system can have at most four stationary states (equilibria), which are derived by setting $\dot{N} = \dot{P} = 0$ in~\eqref{eq:RM1}. In addition to the {\em extinction equilibrium} 
$e_0$, which is stable for $r>0$,
there is a {\em prey-only} equilibrium $e_1(r)$, the {\em Allee equilibrium}  $e_2$, and the 
 {\em coexistence equilibrium} $e_3(r)$, whose stability depends on $r$ and other system parameters:
\begin{equation}\label{eq:RM_equilibria}
e_0=(0,0),~e_1(r)=(r/c,0),~e_2=(\mu,0),~e_3(r)=(N_3,P_3(r)).
\end{equation}
In the above, we include the argument $(r)$ when an equilibrium's position depends on $r$.
The prey and predator densities of the coexistence equilibrium $e_3(r)$ are given by:
$$
N_3 = \frac{\delta \beta}{\chi \alpha -\delta}\ge 0\;\;\mbox{and}\;\;
P_3(r) = \frac{r}{\alpha} \left(1 - \frac{c}{r} N_3\right)\frac{(\beta+N_3)(N_3-\mu)}{\nu + N_3}\ge 0.
$$

The one-parameter bifurcation diagram of the RMA frozen system in Fig.~\ref{fig:1parBD}\bluea \, reveals various bifurcations
and bistability, which are discussed in detail in Sec.~\ref{sec:basin_instability_PP}\ref{sec:bif_analysis}. Most importantly, as $r$ is increased, the coexistence equilibrium $e_3(r)$ undergoes a supercritical Hopf bifurcation $H$, which makes the equilibrium unstable and produces a stable  limit cycle $\Gamma(r)$. The cycle corresponds to {\em oscillatory coexistence of predator and prey} and is the main focus of this study. {  In the ecological literature, this Hopf bifurcation is referred to  as
the paradox of enrichment~\cite{roy2007}.}
As $r$ is increased even further, $\Gamma(r)$  
disappears in a {\em dangerous} heteroclinic bifurcation $h$ at  $r=r_h$, giving
rise to a discontinuity in the branch of coexistence attractors.
Past $r_h$,  the only attractor is the extinction equilibrium $e_0$. {  This heteroclinic bifurcation indicates where complete depletion of the predator becomes part of the cycle. Note that, in the absence of noise,  the predator remains extinct once its level reaches zero because the subspace $\{P=0\}$ is invariant. Hence the counter-intuitive transition to predator extinction at high prey growth rates.}

{  
To show that {  phase} tipping is ubiquitous in predator-prey interactions, we also consider another paradigmatic predator-prey model, the May model~\cite{May2019,strohm:2009}:
\begin{equation}\label{eq:May1}
    \begin{split} 
        \dot{N} &= r(t)\, N\left(1-\frac{c}{r}N\right)\left(\frac{N-\mu}{\nu+N}\right) - \frac{\alpha NP} {\beta+N},\\
        \dot{P} &= sP\left(1-\frac{qP}{N+\epsilon}\right). 
    \end{split}
\end{equation}
This model has the same equation for the prey population density $N$ as the RMA model, but differs in the  equation for the predator population density $P$. Specifically, 
$s$ is the low-density predator growth rate and
$\epsilon$ is introduced to allow prey extinction. In other words, this model assumes that the predator must have access to other prey which allow it to survive at a low density $\epsilon/q$ in the absence of the primary prey $N$.
The parameter $q$ approximates the minimum prey-to-predator biomass ratio that allows predator population growth, and Table~\ref{tab:parValues} in Appendix \ref{app:maymodel} contains realistic parameter values, estimated from Canada lynx and snowshoe hare data \cite{strohm:2009,tyson:2009}. 

In addition to the {\em extinction equilibrium} 
$e_0$, which is always stable,
the May frozen system has a {\em prey-only} equilibrium $e_1(r)$, an {\em Allee equilibrium}  $e_2$, and {\em two coexistence equilibria} $e_3(r)$ and $e_4(r)$, whose stability depends on the system parameters.
Further details and analysis of the May frozen model are provided in Appendix~\ref{app:maymodel}.
}

\subsection{Phase of the cycle}

\begin{figure}[h!]
    \centering 
    \includegraphics[width=1\textwidth]{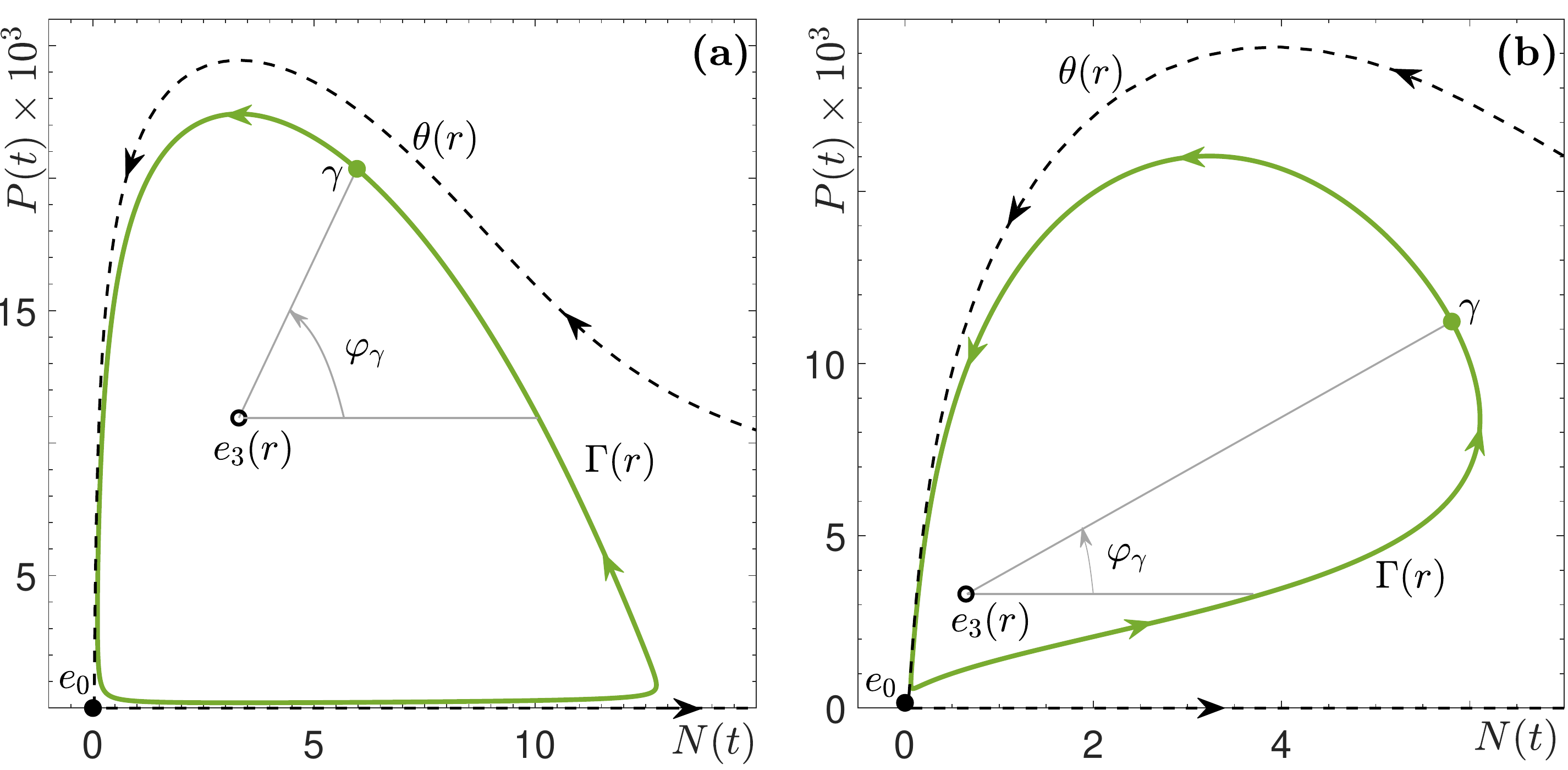}
    \caption{Phase portraits showing the (green) predator-prey limit cycles $\Gamma(r)$ together with their phases $\varphi_\gamma$ and basin boundaries $\theta(r)$ in (a) the autonomous RMA frozen model~\eqref{eq:RM1} with $r = 2.47$ and (b) the autonomous May frozen model~\eqref{eq:May1} with $r = 2$.
    The other parameter values are given in Appendix~\ref{app:maymodel}, Table~\ref{tab:parValues}.  Schematic phase portraits depicting all equilibria and invariant manifolds are shown in  Appendix~\ref{app:maymodel}, Fig.~\ref{fig:SchematicPhase}. 
    }
    \label{fig:phaseOfCycle}
\end{figure}

To depict {  phase} tipping, {  each point on the limit cycle, as well as in a neighbourhood of the cycle,  must be characterised} by its unique phase. 
In the two-dimensional phase space of the autonomous predator-prey frozen systems~\eqref{eq:RM1} and~\eqref{eq:May1}, the stable limit cycle $\Gamma(r)$
makes a simple rotation about the coexistence equilibrium $e_3(r)$. 
We take advantage of this fact and 
assign a unique {phase} $\varphi_{\gamma}\in[0,2\pi)$  to every point $\gamma=(N_\gamma,P_\gamma)$ on the limit cycle using a polar coordinate system anchored in $e_3(r) = (N_{3}(r),P_{3}(r))$:
\begin{align}
\label{eq:phase}
\varphi_{\gamma} = \tan^{-1}\left( 10^3\, \frac{P_{\gamma} - P_{3}}{N_\gamma - N_{3}} \right).
\end{align}
In other words, the phase of the cycle is the angle measured counter-clockwise from the horizontal
half line that extends from $e_3(r)$ in the direction of increasing $N$, as is shown in Fig.~\ref{fig:phaseOfCycle}. Since the values of $P(t)$ for the limit cycles in systems~\eqref{eq:RM1} and~\eqref{eq:May1} are three orders of magnitude smaller than the values of $N(t)$, the ensuing distribution of $\varphi_{\gamma}$ along
$\Gamma(r)$ is highly non-uniform. To address this issue and achieve a uniform distribution of $\varphi_{\gamma}$, we include the factor of $10^3$ in~\eqref{eq:phase}. 

{  In the problem of P-tipping, we often encounter oscillatory solutions that have not converged to the 
limit cycle $\Gamma(r)$.  Equation~\eqref{eq:phase} allows us to define the "phase" of such oscillatory solutions 
in a neighbourhood of $\Gamma(r)$.  
}

\subsection{Climate variability}
\label{sec:climvar}

Climate variability here refers to changes in the state of the climate occurring on year-to-decade time scales.  We model this process by allowing  $r(t)$, i.e., the prey birth rate and the carrying capacity of the ecosystem, to vary over time. This variation can be interpreted as climate-induced changes in resource availability or habitat quality.  Seasonal modelling studies often assume sinusoidal variation in climate parameters~\cite{picoche:2019, ghosh:2019, bessho:2010, hanski:1995}, but many key climate variables vary much more abruptly~\cite{sauve:2020}.  Since our unit of time is years, rather than months, we focus on abrupt changes in climate.\footnote{   In ecology, abrupt changes in the form of a single-switch between two values of an input parameter are called {\em press disturbances}~\cite{schoenmakers2021}.}

Guided by the the approach proposed in~\cite{Marley2020} and~\cite{wilmers2007}, we construct a piecewise constant $r(t)$ using two random processes; see Fig.~\ref{fig:cli_dist_RM_past_h}\bluea. 
First,  we assume the amplitude of $r(t)$ is a random variable with a continuous uniform probability distribution on a closed interval $[r_2,r_1]$.
Second, we assume the number of consecutive years $\ell$ during which the amplitude of $r(t)$ remains constant is a random variable with a 
discrete probability distribution known as the geometric distribution\footnote{
In the statistical literature, the above form of the geometric distribution models the number of failures in a Bernoulli trail until 
the first success occurs, where $\rho$ is the probability of success~\cite{Devroye2006}.}
\begin{align}
\label{eq:gpd}
g(\ell) = \mbox{Pr} (x = \ell) = (1-\rho)^\ell\, \rho,
\end{align}
where $\ell\in\mathbb{Z}_+$ is a positive integer and  $\rho\in(0,1)$.
{  Such an $r(t)$ can be viewed as bounded autocorrelated noise.}
Using actual climate records from four locations in the boreal and deciduous-boreal forest in North America, we choose a realistic value of $\rho = 0.2$~\cite{Marley2020}.
We say the years with constant $r(t)$ are of {high productivity}, or Type-H, if their amplitude is greater than the mean $(r_1+r_2)/2$. Otherwise we say the years are of {low productivity},  or Type-L, as indicated in Fig.~\ref{fig:cli_dist_RM_past_h}\bluea.

\section{
B-tipping vs. P-tipping in oscillatory predator-prey models
}
\label{sec:cipst}

In this section, we use the nonautonomous RMA model~\eqref{eq:RM1} to demonstrate the occurrence of P-tipping in predator-prey interactions. Furthermore, we highlight the counter-intuitive properties of P-tipping by a direct comparison with the intuitive and better understood B-tipping.

{  Note that, in the nonautonomous system,  $e_0$ remains the extinction equilibrium, but the predator-prey limit cycle $\Gamma(r)$ is replaced by (irregular) predator-prey oscillations. Nonetheless, since the system is piecewise autonomous, the dynamics and bifurcations of the autonomous frozen system help us to understand the behaviour of the nonautonomous one.}

\begin{figure}[h!]
    \centering 
    \includegraphics[width=1\textwidth]{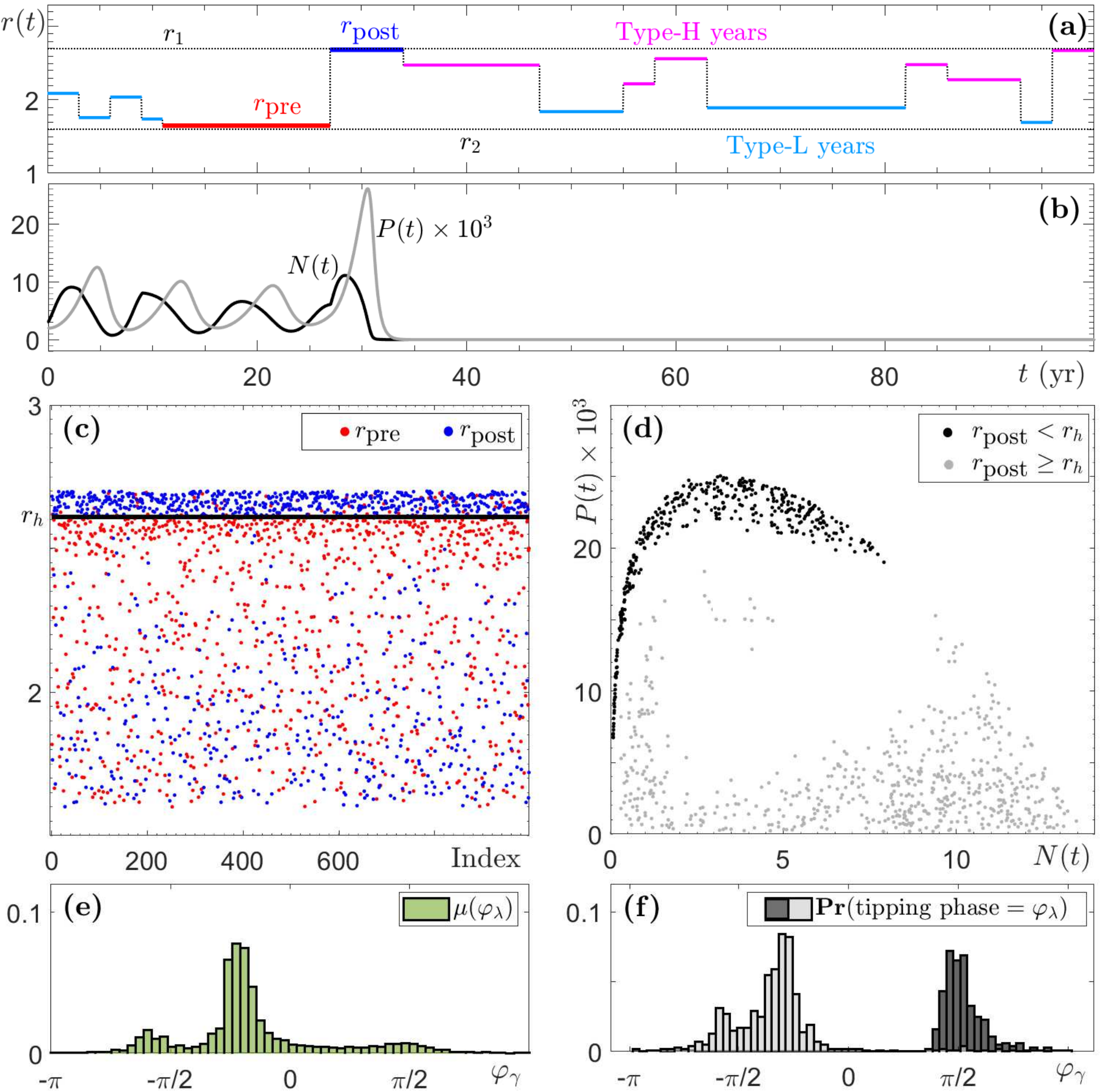}
    \caption{Results of a Monte Carlo simulation for the RMA model~\eqref{eq:RM1}, where time-varying $r(t)$ is generated using $p=0.2$ and ``Climate variability" interval $[r_2,r_1] = [1.6,2.7]$ containing $r_h$. Shown are $10^3$ numerical tipping experiments (B-tipping and P-tipping) for a fixed initial condition $(N_\textrm{0},P_\textrm{0}) = (3,0.002)$. The other parameter values are given in Appendix~\ref{app:maymodel}, Table~\ref{tab:parValues}.
    (a-b) The time profiles of $r(t)$, $N(t)$ and $P(t)$ in a single tipping experiment.
    (c) The values of $r(t)$ (red) {\em pre} and (blue) {\em post} each {  switch that causes a} tipping event.
    (d) States in the $(N,P)$ phase plane 
    {  at the time of the switch that causes a tipping event (i.e. states at the ``tipping time" defined
    in Definition~\ref{defn:Ptip}),
    }
     (gray dots) B-tipping and (black dots) P-tipping.
    (e) The invariant measure $\mu(\varphi_\lambda)$ of the limit cycle $\Gamma(r)$ parameterised by the cycle phase
    $\varphi_\lambda$. 
    (f) Probability distribution 
of tipping phases $\varphi_\lambda$  for (gray) B-tipping and (black) P-tipping.
 }
    \label{fig:cli_dist_RM_past_h}
\end{figure}

\subsection{B-tipping from predator-prey cycles}

We begin with a brief description  
of B-tipping due to the dangerous heteroclinic bifurcation $h$ of the attracting predator-prey limit cycle $\Gamma(r)$.
In the autonomous frozen system, the cycle $\Gamma(r)$ exists for the values of $r$ below $r_h$, 
and disappears in a discontinuous way when $r=r_h$; see  Fig.~\ref{fig:1parBD}\bluea. 
Thus, we expect
{  one obvious} tipping behaviour in the nonautonomous system with a time-varying $r(t)$: 

\begin{itemize}
    \item[(B1)]
    B-tipping from 
     predator-prey {  oscillations}
    to extinction $e_0$ will occur if $r(t)$ increases past the dangerous bifurcation level $r=r_h$, {  and the system
    converges to $e_0$ before switching back to $r<r_h$.}
     \item[(B2)]
     B-tipping will occur from all phases 
      of predator-prey {  oscillations}, but phases where the system spends more time are  more likely to tip. An invariant measure $\mu(\varphi_\gamma)$ of $\Gamma(r)$   
     can be obtained and normalised to 
     {  approximate}
     the probability distribution for B-tipping from a phase $\varphi_{\gamma}$ as shown in Fig.~\ref{fig:cli_dist_RM_past_h}\bluee; see {  Ref.~\cite{halmos1947} and} Appendix~\ref{App:invariantMeasures} for more details on calculating $\mu(\varphi_\gamma)$.
     \item[(B3)]
    B-tipping from 
     predator-prey {  oscillations} cannot occur when $r(t)$ decreases over time because $\Gamma(r)$  does not undergo any dangerous bifurcations upon decreasing $r$.
\end{itemize}

To illustrate properties (B1)--(B3), we perform a Monte Carlo simulation of the nonautonomous RMA system~\eqref{eq:RM1}. We restrict the variation of $r(t)$ to the closed interval $[r_2,r_1]$ containing the bifurcation point $r_h$ (see the "Climate variability" label in Fig.~\ref{fig:1parBD}\bluea, upper arrow), and perform $10^3$ numerical experiments. In each experiment, we start from a fixed initial condition  $(N_0,P_0) = (3,0.002)$ within the basin of attraction of $\Gamma(r)$, and let  $r(t)$ vary randomly as explained in Sec.~\ref{sec:ppm}\ref{sec:climvar}.
We allow the system to continue until tipping from 
{  predator-prey oscillations} to extinction occurs (Fig.~\ref{fig:cli_dist_RM_past_h}\blueb) due to a step change in $r(t)$ from 
 $r_\textrm{pre}$ to $r_\textrm{post}$ 
 (Fig.~\ref{fig:cli_dist_RM_past_h}\bluea). We then record the values of  $r_{\textrm{pre}}$ in red 
 and the values of $r_{\textrm{post}}$ in blue in 
 Fig.~\ref{fig:cli_dist_RM_past_h}\bluec, the state 
 in the $(N,P)$ phase space 
 {  when the switch from $r_{\textrm{pre}}$ to $r_{\textrm{post}}$} occurs in
 Fig.~\ref{fig:cli_dist_RM_past_h}\blued, 
 and the corresponding phase of 
 {  this state} to produce the tipping-phase histograms in 
 Fig.~\ref{fig:cli_dist_RM_past_h}\bluef.
B-tipping is identified as the blue dots above $r=r_{h}$ in 
Fig.~\ref{fig:cli_dist_RM_past_h}\bluec, meaning that
transitions to extinction occur when $r(t)$ changes from $r_\textrm{pre}<r_h$ 
to $r_\textrm{post} > r_h$ in agreement with (B1)
and (B3). The tipping phases corresponding to grey dots in Fig.~\ref{fig:cli_dist_RM_past_h}\blued, 
and the ensuing grey histogram in Fig.~\ref{fig:cli_dist_RM_past_h}\bluef,  correlate 
almost perfectly with the green invariant measure $\mu(\varphi_\gamma)$ 
of $\Gamma(r)$  in Fig.~\ref{fig:cli_dist_RM_past_h}\bluee, in agreement with (B2).

\subsection{P-tipping from predator-prey cycles}

The most striking result of the simulation is that B-tipping is not the only tipping mechanism at play. It turns out that there are other, unexpected and counter-intuitive tipping transitions. These transitions indicate a new tipping mechanism, whose dynamical properties are in stark contrast to B-tipping:
\begin{itemize}
    \item[(P1)]
    Tipping from the predator-prey 
    {  oscillations} to extiction occurs when $r(t)$ decreases and does not cross any dangerous bifurcations of $\Gamma(r)$, which is in contrast to (B1) and (B3). This is evidenced in 
    Fig.~\ref{fig:cli_dist_RM_past_h}\bluec~by
    the blue dots below $r=r_{h}$ depicting transitions to extinction when $r(t)$ changes from $r_\textrm{pre} < r_h$ to
    $r_\textrm{post} < r_\textrm{pre}$.
     \item[(P2)]
     Tipping occurs {\em only from  certain phases} of 
      predator-prey {  oscillations}, which is in contrast to (B2). This is evidenced by the  black dots in Fig.~\ref{fig:cli_dist_RM_past_h}\blued, and the ensuing black tipping-phase histogram in Fig.~\ref{fig:cli_dist_RM_past_h}\bluef.
     \item[(P3)]
    The tipping phases do not correlate at all with the  invariant measure $\mu(\varphi_\gamma)$ of $\Gamma(r)$ shown in Fig.~\ref{fig:cli_dist_RM_past_h}\bluee. This is evidenced by a comparison 
    with the black histogram in Fig.~\ref{fig:cli_dist_RM_past_h}\bluef.
\end{itemize}  
Since the unexpected tipping transitions occur only from certain phases of 
{  predator-prey oscillations},
we refer to this phenomenon as {\em {  phase} tipping} or {\em P-tipping}. 

Although P-tipping is less understood than B-tipping, it is ubiquitous and possibly even more relevant for predator-prey interactions. In Fig.~\ref{fig:cli_dist_RM} we restrict  climate variability in the RMA model~\eqref{eq:RM1} to a closed interval $[r_2,r_1]$ that does not contain $r_h$. In other words, we set $r_1 < r_h$. Since the time-varying input $r(t)$ cannot cross the dangerous heteroclinic bifurcation, all tipping transitions are P-tipping events. 
Furthermore, owing to the absence of dangerous bifurcations of $\Gamma(r)$ in the May model~\eqref{eq:May1} in Fig.~\ref{fig:1parBD}\blueb, P-tipping from 
{  predator-prey oscillations} to extinction $e_0$ is the only tipping mechanism in Fig.~\ref{fig:cli_dist_May}. Note that P-tipping is more likely to occur in the May model, as evidenced by shorter tipping times; compare Figs.~\ref{fig:cli_dist_RM}\bluec\, and \ref{fig:cli_dist_May}\bluec.

The numerical experiments in Figs.~\ref{fig:cli_dist_RM}~and~\ref{fig:cli_dist_May} serve as motivating examples for the development of a general mathematical framework for P-tipping 
in Section~\ref{sec:basin_instability}.

\begin{figure}[t]
    \centering 
    \includegraphics[width=1\textwidth]{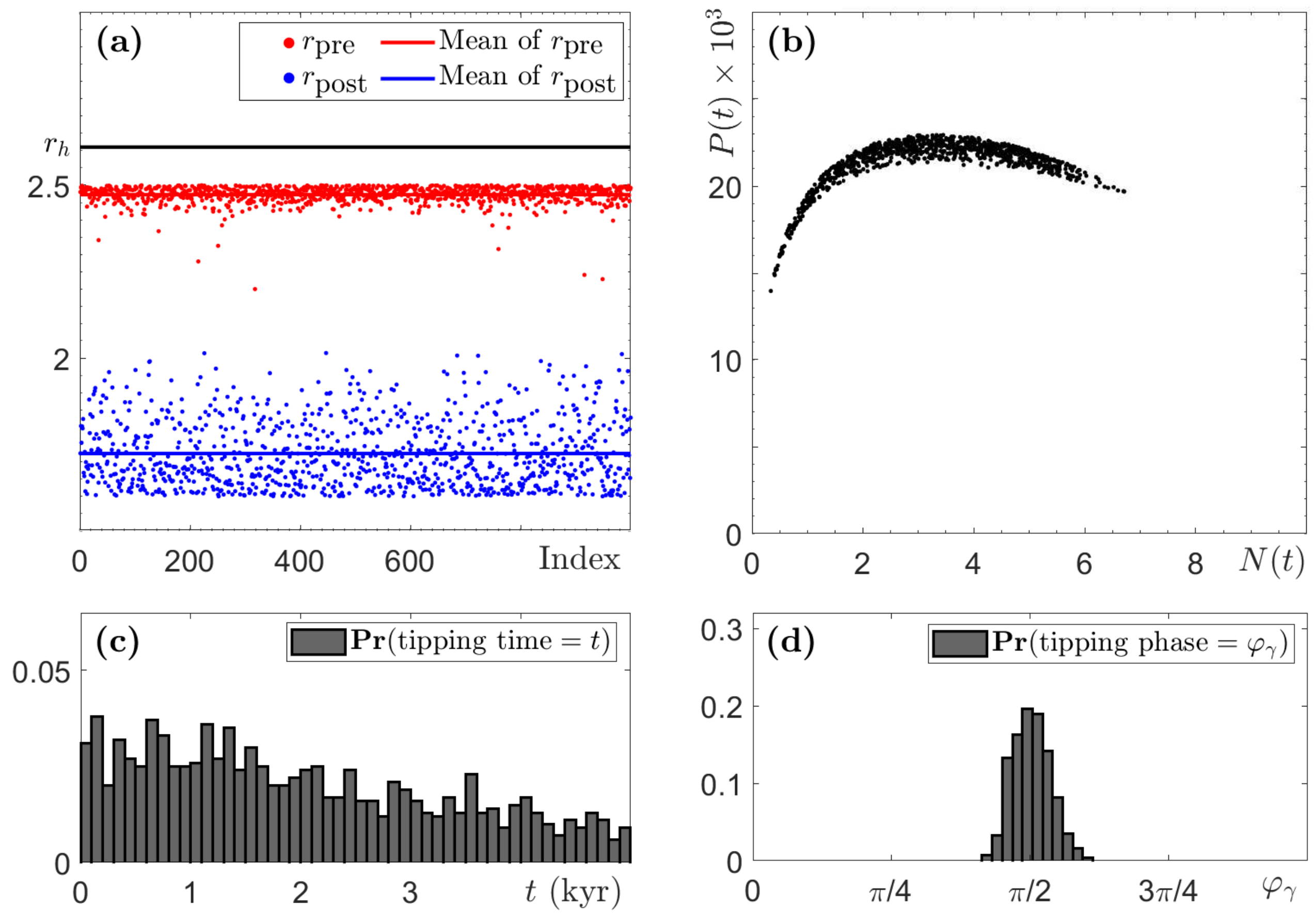}
    \caption{ 
(a-b) and (d) The same as in Fig.~\ref{fig:cli_dist_RM_past_h} except for 
$r(t)$ taking values from a different ``Climate variability" interval $[r_2,r_1] = [1.6,2.5]$ that does not contain $r_h$. As a result, each of the 1000 tipping events is P-tipping. (c) The probability distribution of tipping at time $t$. The other parameter values are given in Appendix~\ref{app:maymodel}, Table~\ref{tab:parValues}.
}
    \label{fig:cli_dist_RM}
\end{figure}

\subsection{The Allee threshold: Intuitive explanation of P-tipping}

{ 
Intuitively, P-tipping from predator-prey oscillations to extinction in the nonautonomous system can  be understood 
in terms of an {\em Allee threshold} $\theta(r)$ in the autonomous frozen system,
separating trajectories that lead to extinction 
from those that approach the predator-prey cycle (see Figs.~\ref{fig:phaseOfCycle} and~\ref{fig:SchematicPhase}), and 
how a given drop in prey resources $r(t)$ affects different phases near the predator-prey cycle via the changing Allee threshold.

The shape and position of both the Allee threshold 
$\theta(r)$ and the predator-prey cycle $\Gamma(r)$ 
are modified by a drop  in  prey resources $r(t)$.
The strongest impact is expected when the drop coincides 
with the region of the fastest decline in prey $N(t)$ and a large 
predator population $P(t)$. These situations occur near the part of the cycle within a range of phases around $\varphi_{\gamma}=\pi/2$, which is close to $\theta(r)$. 
There, the drop speeds up the prey decline, which, in conjunction with high predation 
pressure, creates perfect conditions for the ecosystem to move away from the modified cycle, cross the even closer modified Allee threshold, and move towards extinction.
Indeed, Figs.~\ref{fig:cli_dist_RM} and~\ref{fig:cli_dist_May} show 
that P-tipping occurs from a range of phases around $\varphi_{\gamma}=\pi/2$. 
The ecosystem response is very different if the same drop in prey 
resources coincides with the region of the fastest growth of 
prey $N(t)$ and a small predator population $P(t)$. These situations occur near a different part of the cycle,
within a range of phases around $\varphi_{\gamma}=-\pi/2$, which is away from $\theta(r)$.
There, the drop slows or even reverses the prey growth, but low predation pressure prevents the ecosystem from crossing the distant Allee threshold and helps it adapt 
to the modified cycle instead. Hence the observed phase sensitivity of tipping from predator-prey oscillations to extinction in the nonautonomous systems.
}

\begin{figure}[t]
    \centering 
    \includegraphics[width=1\textwidth]{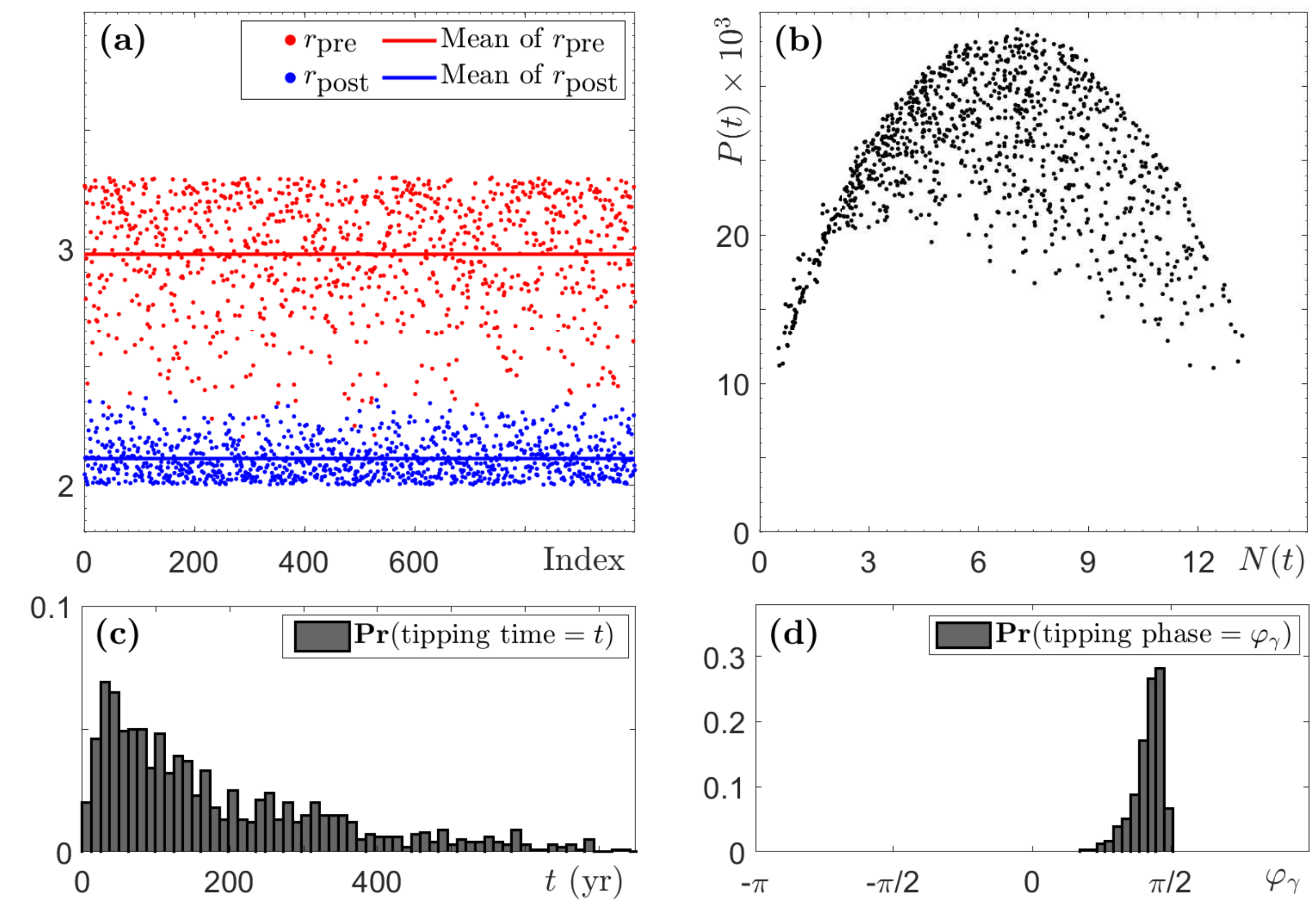}
    \caption{
    (a-d) The same as in Fig.~\ref{fig:cli_dist_RM} but for the
    May frozen model~\eqref{eq:May1} with $r(t)$ taking values from the ``Climate variability" interval $[r_2,r_1] = [2,3.3]$. Each of the 1000 tipping events is an instance of P-tipping. The other parameter values are given in Appendix~\ref{app:maymodel}, Table~\ref{tab:parValues}.
    }
    \label{fig:cli_dist_May}
\end{figure}

\section{A geometric framework for  P-tipping: Partial basin instability}
\label{sec:basin_instability}
Motivated by the numerical experiments in Figs.~\ref{fig:cli_dist_RM}~and~\ref{fig:cli_dist_May},
and the fact that P-tipping is not captured by classical bifurcation theory, the aim of this section is to provide mathematical tools for analysis of P-tipping.
Specifically, we develop a simple geometric framework that uses global properties of the autonomous frozen system to study P-tipping from attracting limit cycles {  and their neighbourhoods} in the nonautonomous system.
The key concept is {\em basin instability}\footnote{  
Not to be confused with the notion of "basin stability"
introduced as a measure related to the volume of the basin of attraction \cite{menck2013}. 
}. This concept was first introduced in~\cite[Section~5.2]{o2019}
to 
study irreversible R-tipping from base states that are stationary (attracting equilibria) for fixed-in-time external inputs. Here, we extend this concept to base states that are attracting limit cycles for fixed-in-time external  inputs. Our framework will allow us to  give easily testable criteria for the occurrence of P-tipping from limit cycles in general, and explain the counter-intuitive collapses to extinction in the predator-prey systems from Sec.~\ref{sec:cipst}.

To define basin instability  {  and P-tipping for limit cycles in general terms,
we consider an $n$-dimensional nonautonomous system
\begin{equation}
    \label{eq:general}
    \dot{x} = f(x,p(t)),
\end{equation}
with $x\in\mathbb{R}^n$, and a piecewise constant external input $p(t)$ that can be single-switch or multi-switch.
When it is important to highlight the dependence of multi-switch inputs on $\rho$ \big(see Eq.~\eqref{eq:gpd}\big),  we write $p_{\rho}(t)$ instead of $p(t)$.
Note that the RMA~\eqref{eq:RM1} and May~\eqref{eq:May1} models with $r(t)$ from Section~\ref{sec:ppm}\ref{sec:climvar}
are examples of~\eqref{eq:general}. Furthermore, we write
$$
x(t,x_0;t_0)
$$
to denote a solution to the nonautonomous system~\eqref{eq:general} 
at time $t$ started from $x_0$ at initial time $t_0$.
We also consider the corresponding autonomous frozen system with different but fixed-in-time values of the external input $p$,
and write
$$
x(t,x_0;p)
$$
to denote a solution to the autonomous frozen system 
at time $t$ started from $x_0$ for a fixed $p$.
}

\subsection{Ingredients for defining basin instability}

One key ingredient of a basin instability definition is the
{\em base attractor} in the autonomous frozen system, denoted $\Gamma(p)$, whose shape and position in the phase space vary  with the input parameter(s) $p$.
The second key ingredient is the {\em basin of attraction of the base attractor}, denoted $B(\Gamma,p)$, whose shape and extent may also vary with the input parameter(s) $p$.
For non-stationary attractors $\Gamma(p)$, we work
with the distance{ \footnote{  The distance between $x(t,x_0;p)$  and $\Gamma(p)$ is
$d\left[x(t,x_0;p),\Gamma(p)\right] = \inf_{\gamma\in\Gamma(p)} \Vert x(t,x_0;p) - \gamma \Vert$.}
between a solution $x(t,x_0;p)$ 
and the set $\Gamma(p)$, and write
$$
x(t,x_0,p)\to\Gamma(p)
\quad\mbox{as}\quad t \to +\infty,
$$
when this distance tends to zero as $t \to +\infty$.}
We define $B(\Gamma,p)$ as the open set of all $x_0$ whose trajectories converge to $\Gamma(p)$ forward in time:
$$
B(\Gamma,p) = \left\{x_0\;:\;
{  x(t,x_0,p)\to\Gamma(p)}\;\;\textrm{as}\;\; t\to +\infty
\right\}.
$$
We often refer to {\em the closure of the basin of attraction} of $\Gamma(p)$, denoted $\overline{B(\Gamma,p)}$, which comprises $B(\Gamma,p)$ and its boundary, and to the {\em basin boundary} of $\Gamma(p)$, which is  given by the set difference $\overline{B(\Gamma,p)}\setminus B(\Gamma,p)$.
Additionally, we assume that either all or part of the basin boundary of $\Gamma(p)$ is a
basin boundary of at least two attractors. This property, in turn, requires that
the autonomous frozen system is at least {\em bistable}, meaning that it has at least one more attractor, other that $\Gamma(p)$, for the same values of the input parameter(s) $p$.

The third key ingredient is a 
{  {\em parameter path} $\Delta_p$, which we define as  a connected set of all possible values of the external input $p(t)$.}
It is important that $\Delta_p$ does not cross any classical autonomous bifurcations of the base attractor $\Gamma(p)$.

\subsection{Definitions of basin instability for limit cycles}

In short, {\em basin instability} of the base attractor on a parameter path describes the position of the base attractor at some point on the path relative to the position of its basin of attraction at other points on the path. 
Here, we define this concept rigorously for attracting limit cycles setwise.
\begin{definition} 
\label{def:partialBI_set}
{  Consider a parameter path $\Delta_p$. Suppose the frozen system has a family of hyperbolic attracting limit cycles $\Gamma(p)$ that vary $C^1$-smoothly with $p\in\Delta_p$.}
We say $\Gamma(p)$ is {\em basin unstable} on a path $\Delta_p$ if there are two points on the path, $p_1,p_2 \in \Delta_p$, such that the limit cycle $\Gamma(p_1)$ is not contained in the basin of attraction of $\Gamma(p_2)$:
  \begin{equation}
  \label{eq:BI_set}
\textrm{There exist}\;\; p_1,p_2 \in \Delta_p \;\; \textrm{ such that}\;\;\Gamma(p_1) \not\subset {B(\Gamma,p_2)}.  
    \end{equation}

Furthermore, we distinguish  two {\em observable (or typical)} cases of basin instability:
\begin{itemize}
    \item[(i)] 
    We say $\Gamma(p)$ is {\em partially basin unstable} on a path $\Delta_p$ if there are two points on the path, $p_1$ and $p_2 \in \Delta_p$, such that the limit cycle $\Gamma(p_1)$ is not fully contained in the closure of the basin of attraction of $\Gamma(p_2)$, and, for every two points on the path, $p_3$ and $p_4 \in \Delta_p$,  $\Gamma(p_3)$ has a non-empty intersection with the basin of attraction of $\Gamma(p_4)$:
    \begin{equation}
    \begin{split}
    & \textrm{There exist}\;\; p_1,p_2 \in \Delta_p \;\; \textrm{ such that}\;\;\Gamma(p_1) \not\subset \overline{B(\Gamma,p_2)}, \;\; \textrm{and}\\
  &\Gamma(p_3) \bigcap {B(\Gamma,p_4)} \neq \emptyset\;\;\textrm{for every}\;\; p_3, p_4\in\Delta_p.
  \end{split}
  \end{equation}

    \item[(ii)] We say $\Gamma(p)$ is {\em totally basin unstable} on a path $\Delta_p$ if there are (at least) two points on the path,
     $p_1$ and $p_2 \in \Delta_p$, such that $\Gamma(p_1)$ lies outside the closure 
     of the basin of attraction of $\Gamma(p_2)$:
    \begin{align}
\textrm{There exist}\;\; p_1,p_2 \in \Delta_p \;\; \textrm{ such that}\;\; \Gamma(p_1) \bigcap \overline{B(\Gamma,p_2)} = \emptyset.
    \end{align}
    \end{itemize}
\end{definition}
\begin{remark}
\label{rem:partialBI_set}
Additionally, there are two {\em indiscernible (or special)} cases of basin instability for limit cycles. They cannot be easily distinguished by observation from total basin instability, or from lack of basin instability. 
However, the indiscernible cases are necessary (although not sufficient) for the onset of partial basin instability and for transitions between partial and total basin instability.
    \begin{itemize}
    \item[(iii)] 
     We say $\Gamma(p)$ is {\em marginally basin unstable} on a path $\Delta_p$ if, in addition to \eqref{eq:BI_set}, for every two points on the path, $p_3$ and $p_4 \in \Delta_p$, the limit cycle $\Gamma(p_3)$ is contained in $\overline{B(\Gamma,p_4)}$:
    \begin{align}
     \Gamma(p_3) \subset \overline{B(\Gamma,p_4)}
    \;\;\textrm{for every}\;\; p_3, p_4\in\Delta_p.
    \end{align}
    The special case of marginal basin instability separates the typical cases of ``no basin instability" and ``partial basin instability". Furthermore, it is related to ``invisible R-tipping'' and to transitions between ``tracking'' and ``partial R-tipping''  identified in~\cite{alkhayuon2018}. 
    \item[(iv)] 
    We say $\Gamma(p)$ is {\em almost totally basin unstable} on a path $\Delta_p$ if 
    there are (at least) two points on the path, $p_1$ and $p_2 \in \Delta_p$, such that $\Gamma(p_1)$ does not intersect $B(\Gamma,p_2)$, and,
    for every two points on the path, $p_3$ and $p_4 \in \Delta_p$, the limit cycle $\Gamma(p_3)$ intersects $\overline{B(\Gamma,p_4)}$:
    \begin{equation}
    \label{eq:almostTBU}
    \begin{split}
    & \textrm{There exist}\;\; p_1, p_2\in\Delta_p \;\;\textrm{such that}\;\; \Gamma(p_1) \bigcap B(\Gamma,p_2) = \emptyset, \;\;\textrm{and}\\
    & \Gamma(p_3) \bigcap \overline{B(\Gamma,p_4)} \ne \emptyset  \;\;\textrm{for every}\;\; p_3, p_4\in\Delta_p. 
    \end{split}
    \end{equation}
    The special case of almost total basin instability separates  the typical cases of ``partial basin instability" and ``total basin instability". Furthermore,
    it is related to transitions between ``partial R-tipping" and "total R-tipping" described in~\cite{alkhayuon2018}. 
    \end{itemize}
Note that, for equilibrium base states, ``partial basin instability" is not defined, whereas ``marginal basin instability" and ``almost total basin instability" become the same condition.
\end{remark}

Guided by the approach proposed in~\cite{o2019}, we would like to
augment the classical autonomous bifurcation diagrams for the
autonomous frozen system  with information about (partial) basin instability of the base attractor $\Gamma(p)$. The aim is to reveal {  nonautonomous instabilities} that {  cannot be explained by classical autonomous bifurcations of the frozen system.}
To illustrate  basin instability of $\Gamma(p)$ in the bifurcation diagram  of the autonomous frozen system, we define the {\em region of  basin instability} of $\Gamma(p)$ in the space of the input parameters as follows:
\begin{definition} 
\label{def:BI}
In the autonomous frozen system, consider 
{  a $C^1$-smooth family of hyperbolic attracting limit cycles $\Gamma(p)$, and denote it with $G$. 
For a fixed $p=p_1$, we define a {\em region of basin instability} of $\Gamma(p_1)\in G$
as a set of all points $p_2$ in the space of the input parameters $p$, such that $\Gamma(p_1)$ is not contained 
 in the basin of attraction of $\Gamma(p_2)\in G$:}
\begin{align}
\label{eq:BIR}
BI(\Gamma,p_1) := 
\left\{
p_2\;:\; \Gamma(p_1) \not\subset 
{B(\Gamma,p_2)}
{  \;\;\mbox{and}\;\;
\Gamma(p_2)\in G}
\right\}.
\end{align}
\end{definition}

\subsection{Partial basin Instability and P-tipping}
\label{sec:pbipt}
{ 
Thus far, we have worked with a loosely defined concept of P-tipping. In this section, we give rigorous definitions of P-tipping for single-switch and multi-switch $p(t)$,
show that partial basin instability of $\Gamma(p)$ for a single-switch $p(t)$ is necessary and sufficient for the occurrence of P-tipping from $\Gamma(p)$, and discuss the applicability of this result to multi-switch $p(t)$. 
}
{ 
\begin{definition} 
\label{defn:Ptip}
Consider a nonautonomous system~\eqref{eq:general} with a piecewise constant input $p(t)$ on a parameter path $\Delta_p$.
Suppose the autonomous frozen system has a 
 family of hyperbolic attracting limit cycles $\Gamma(p)$ that vary $C^1$-smoothly with $p\in\Delta_p$.
\begin{itemize}
    \item[(i)] 
    Suppose $p(t)$ is a single-switch that changes from $p_1\in\Delta_p$ to $p_2\in\Delta_p$ at time $t=t_1$. Suppose also the system is on $\Gamma(p_1)$ at $t=t_1$. We then say that system~\eqref{eq:general} undergoes {\em irreversible P-tipping} from $\Gamma(p_1)$ if there are $x_{a}, x_{b}\in\Gamma(p_1)$, such that 
    $$
    x(t,x_a;p_2)\to\Gamma(p_2)
    \quad\mbox{as}\quad t\to +\infty
    \quad\mbox{and}\quad
    x(t,x_b;p_2)\notin \overline{B(\Gamma,p_2)}
    \quad\mbox{for all}\quad
    t > t_1. 
    $$
    We call $\varphi_{x_b}$ a {\em tipping  phase} associated with each such $x_b$.
    \item[(ii)] 
    Suppose $p_{\rho}(t)$ is multi-switch with a fixed $\rho$. 
    If $x(t_,x_0;t_0)$ leaves the basin of attraction $B(\Gamma,p_{\rho}(t))$ for good, we use $t_1$ to denote the smallest switching time such that
    $$
    x(t,x_0;t_0)\notin \overline{B(\Gamma,p_{\rho}(t))}
    \quad\mbox{for all}\quad
    t > t_1.
    $$
    We use $x_b = x(t_1,x_0;t_0)$ to denote the corresponding state, and $\varphi_{x_b}$ to denote the corresponding {\em tipping phase}.
    We then say that system~\eqref{eq:general} undergoes {\em irreversible P-tipping} if, for 
    some initial condition $x_0\in B(\Gamma,p_{\rho}(t_0))$ and all realisations of $p_{\rho}(t)$, there are  tipping phases $\varphi_{x_b}$ and also a non-zero Lebesgue measure subset of $[0,2\pi)$ that does not contain any tipping phases $\varphi_{x_b}$.
\end{itemize}
We call $t_1$ the {\em tipping time}
\end{definition}
}
\begin{remark}
\label{rem:various}
{  It should be possible to extend Definition~\ref{defn:Ptip} to:}
\begin{itemize}
{ 
    \item [(i)]
     Smoothly varying $p(t)$, for which P-tipping from $\Gamma(p)$ is expected to depend on the rate of change of $p(t)$~\cite{alkhayuon2018,longo2020}. 
     \item [(ii)]
     Non-periodic attractors such as tori or chaotic attractors, which may require an alternative phase definition. We return to this point in Section~\ref{sec:conclusions}.
     }
\end{itemize}
\end{remark}

{ 
In general, the occurrence of P-tipping depends on the initial state, the properties of the external input $p(t)$, and the topological structure of the phase space. We now show that partial basin instability of $\Gamma(p)$ for a single-switch $p(t)$ is necessary and sufficient for the occurrence of P-tipping from $\Gamma(p)$.
}
\begin{proposition} 
\label{Prop:P-tipping}
{ 
Consider a nonautonomous system~\eqref{eq:general} 
and a parameter path $\Delta_p$.
Suppose the frozen system has  a 
 family of hyperbolic attracting limit cycles $\Gamma(p)$ that vary $C^1$-smoothly with $p\in\Delta_p$, and $\Gamma(p)$ is partially basin unstable on $\Delta_p$.}
Then, for all $p_1$ and $p_2\in\Delta_p$, {  a single-switch parameter change} from $p_1$ to $p_2$ gives {\em irreversible P-tipping} from $\Gamma(p_1)$ if and only if $\Gamma(p_1) \not\subset \overline{B(\Gamma,p_2)}$.
\end{proposition}

\begin{proof}
{  A single-switch} parameter change  from $p_1$ to $p_2$ at time $t=t_0$ reduces the
problem to an autonomous initial value problem with initial condition
$x_0=x(t_0)$ and fixed
$p=p_2$. It follows from the definition of basin of attraction that only solutions $x(t,x_0;p_2)$ started from 
$x_0 \in B(\Gamma,p_2)$ are attracted to the limit cycle $\Gamma(p_2)$. 
Thus, if  $\Gamma(p)$ is partially basin unstable on $\Delta_p$ and  $\Gamma(p_1) \not\subset \overline{B(\Gamma,p_2)}$, then there will be  $\gamma \in \Gamma(p_1)\setminus\overline{B(\Gamma,p_2)}$, that give irreversible tipping,
and  $\gamma \in \Gamma(p_1) \cap B(\Gamma,p_2)$, that give no tipping.
Conversely, if there is irreversible P-tipping from $\Gamma(p_1)$, then there must be $\gamma \in \Gamma(p_1)\setminus\overline{B(\Gamma,p_2)}$, which implies
$\Gamma(p_1) \not\subset \overline{B(\Gamma,p_2)}$.
\end{proof}
{  
This rigorous statement no longer holds for multi-switch piecewise constant inputs $p_{\rho}(t)$. The reason is that trajectories are no longer guaranteed to converge to the limit cycle $\Gamma(p)$,  or to the alternative attractor of the frozen system, 
if the time interval between consecutive switches is short compared to the time of convergence. Additionally, trajectories started in the basin of attraction of $\Gamma(p)$ may move away from $\Gamma(p)$ for finite time.
These differences allow for two dynamical scenarios that cannot occur in a system that starts on $\Gamma(p)$ and is subject to a single-switch $p(t)$.

In the first scenario, following a switch, the system leaves the basin of attraction of $\Gamma(p)$, but fails to converge to an alternative attractor before the next
switch happens, re-enters the basin of attraction of $\Gamma(p)$ upon the second switch, and avoids P-tipping in spite of basin instability of $\Gamma(p)$. We refer to such events as ``rescue events"\cite{Marley2020}. 
Hence, basin instability of $\Gamma(p)$
for a given switch within a mulit-switch $p(t)$
does not guarantee the occurrence of tipping 
upon this particular switch.
For the second scenario, we extend the concept of partial basin instability to the whole basin of attraction of $\Gamma(p)$. 
Suppose that $\Gamma(p)$ is basin stable on $\Delta_p$, but its basin of attraction is partially basin unstable on $\Delta_p$. Following a switch, the trajectory  
moves away from $\Gamma(p)$ and enters the basin unstable part of the basin of attraction of $\Gamma(p)$, then the next switch happens, and the system undergoes P-tipping in the absence of basin instability of $\Gamma(p)$. 
Hence, partial basin instability of $\Gamma(p)$ need not be necessary for the occurrence of P-tipping with multi-switch $p(t)$.

Keeping in mind that multi-switch P-tipping is defined for all realisations of $p_\rho(t)$, it should be possible to show that, for multi-switch piecewise constant $p(t)$:
\begin{itemize}
    \item 
    Partial basin instability of $\Gamma(p)$ on $\Delta_p$ is sufficient for the occurrence of P-tipping in system~\eqref{eq:general}.
    \item
    If $p_{\rho}(t)$ allows trajectories to converge to $\Gamma(p)$ between all consecutive switches, then  partial basin instability of $\Gamma(p)$ on $\Delta_p$ is necessary and sufficient for the occurrence of P-tipping in system~\eqref{eq:general}.
\end{itemize}
}

\section{Partial basin instability and P-tipping in predator-prey models}
\label{sec:basin_instability_PP}

In this section, we start with classical autonomous bifurcation analysis of the predator-prey frozen systems~\eqref{eq:RM1}~and~\eqref{eq:May1} to identify parameter regions with bistability between  predator-prey cycles and extinction. Then, we show that predator-prey cycles can be partially basin unstable on several parameter paths $\Delta_r$ that lie within these regions of bistability. Finally, we demonstrate that  partial basin instability of predator-prey cycles on a path $\Delta_r$ explains
{ 
the counter-intuitive collapses to extinction that occur only from  certain phases of predator-prey oscillations}, and gives simple testable criteria for the occurrence of P-tipping in the nonautonomous  predator-prey system.

\subsection{Classical bifurcation analysis: Limit cycles
and bistability}
\label{sec:bif_analysis}

There are four ecologically relevant parameter regions in the predator-prey frozen systems~\eqref{eq:RM1} and~\eqref{eq:May1}, shown in Figure~\ref{fig:2parBD}. These regions have qualitatively different dynamics that  can be summarised in terms of stable states as follows:
\begin{itemize}
    \item {\em  Oscillatory Coexistence or Extinction:} The system is bistable and can either settle at the extinction equilibrium $e_0$, or  self-oscillate as it converges to the predator-prey limit cycle $\Gamma(r)$. Here is where P-tipping may occur; see the green regions in Fig.~\ref{fig:2parBD}.
    \item {\em  Stationary Coexistence or Extinction:}  The system is bistable and can settle  either at the extinction equilibrium $e_0$, or at the coexistence equilibrium $e_3(r)$; see the yellow regions in Fig.~\ref{fig:2parBD}.
    \item {\em Prey Only or Extinction:} The system is bistable and can settle either at the extinction  equilibrium $e_0$, or at the prey-only equilibrium $e_1(r)$; see the upper pink region in Fig.~\ref{fig:2parBD}\bluea.
    \item {\em Extinction:} The system is monostable and can only settle at the extinction equilibrium $e_0$; see the other pink regions in Fig.~\ref{fig:2parBD}.
\end{itemize}
The region boundaries are obtained via two-parameter bifurcation analysis using the numerical continuation software XPPAUT~\cite{ermentrout2002}. This analysis extends our discussion of the one-parameter bifurcation diagrams from Fig.~\ref{fig:1parBD}.
{  We refer to Appendix~\ref{app:cba} for the details of the bifurcation analysis, and to~\cite{kuznetsov2013} for more details on classical autonomous bifurcation theory.
}

\begin{figure}[t]
\centering 
{\includegraphics[width=1\linewidth]{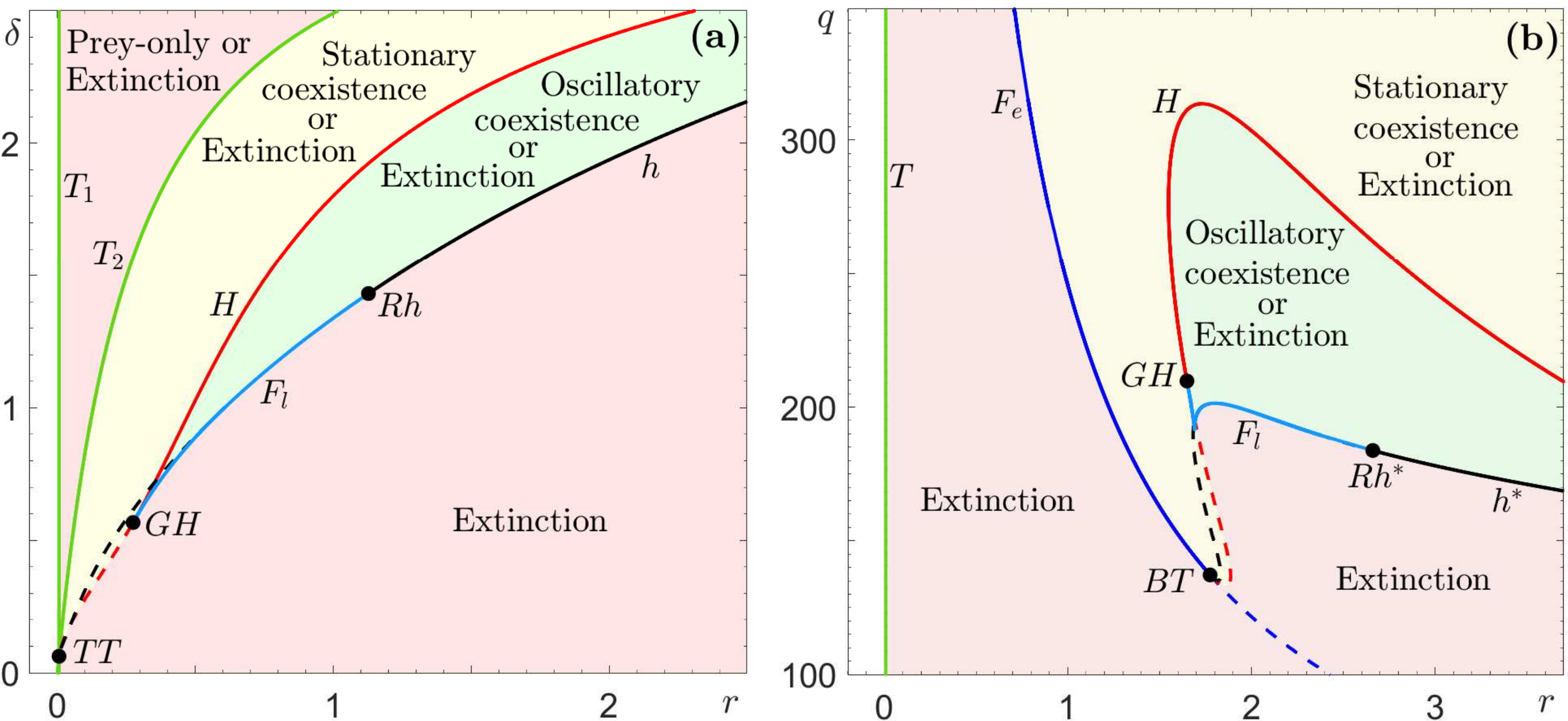}}
\caption{Two-parameter bifurcation diagrams for (a) the autonomous RMA frozen model~\eqref{eq:RM1} in the $(r,\delta)$ parameter plane, and (b) the autonomous May frozen model~\eqref{eq:May1} in the 
$(r,q)$ parameter plane. The other parameter values are given in Appendix~\ref{app:maymodel}, Table~\ref{tab:parValues}.
}
\label{fig:2parBD}
\end{figure}

\begin{figure}[t]
    \centering 
    \includegraphics[width=1\textwidth]{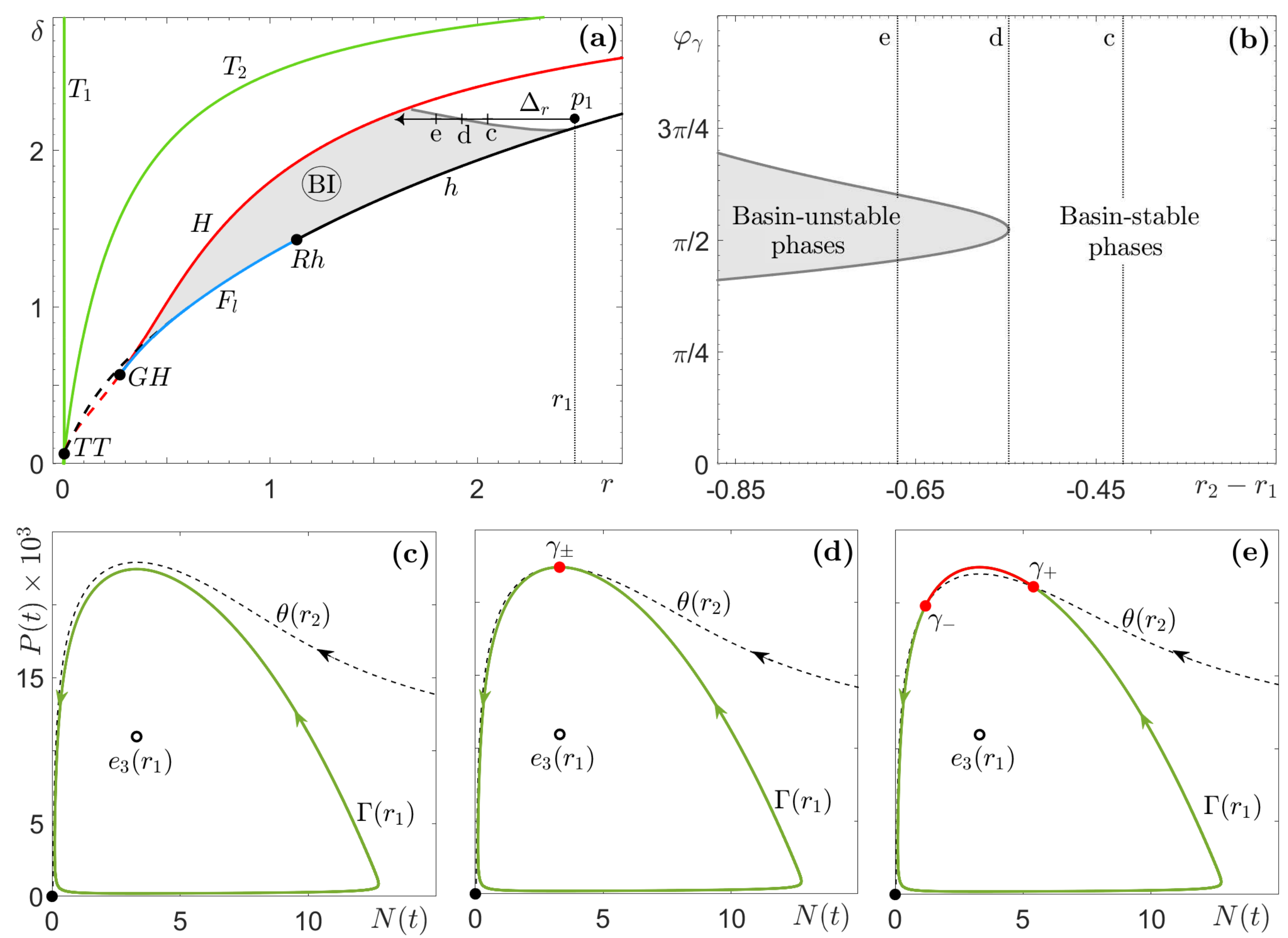}
    \caption{
    (a) The two-parameter bifurcation diagram for 
    the autonomous RM frozen model~\eqref{eq:RM1} from Fig.~\ref{fig:2parBD}\bluea\,with the addition of the (grey) region of partial basin instability, $BI(\Gamma,p_1)$ for $p_1 = (2.47,2.2)$, as defined by~\eqref{eq:BIR}, and the parameter path $\Delta_r$ from $p_1$.
    (b) The range of basin unstable phases for the predator-prey limit cycle $\Gamma(r)$ along $\Delta_r$.
    (c)--(e) Selected $(N,P)$ phase portraits showing (c) no basin instability for $r_2 = 2.05$, (d) marginal basin instability for $r_2 = 1.923$, and (e) partial basin instability of $\Gamma(r)$ on $\Delta_r$ for $r_2 = 1.8$. Basin stable parts of  $\Gamma(r)$ are shown in green, basin unstable parts of $\Gamma(r)$ are shown in red.
    The other parameter values are given in Appendix~\ref{app:maymodel}, Table~\ref{tab:parValues}.
    }
    \label{fig:BI_RM}
\end{figure}

\begin{figure}[t]
    \centering 
    \includegraphics[width=1\textwidth]{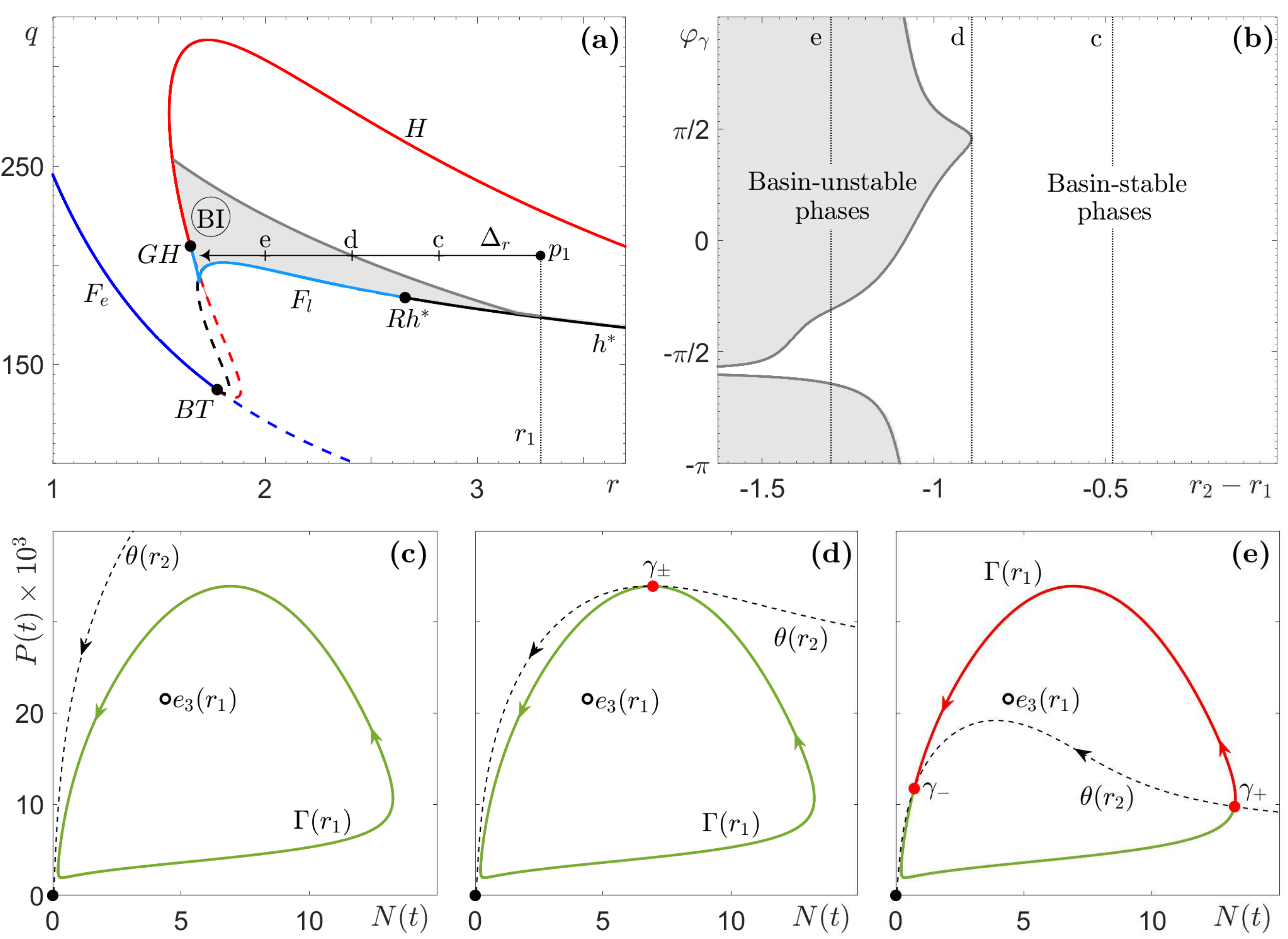}
    \caption{
    (a) The two-parameter bifurcation diagram for 
    the autonomous May frozen model~\eqref{eq:May1} from Fig.~\ref{fig:2parBD}\blueb\,with the addition of the (grey) region of partial basin instability, $BI(\Gamma,p_1)$ for $p_1 = (3.3,205)$, as defined by~\eqref{eq:BIR}, and the parameter path $\Delta_r$ from $p_1$.
    (b) The range of basin unstable phases for the predator-prey limit cycle $\Gamma(r)$ along $\Delta_r$.
    (c)--(e) Selected $(N,P)$ phase portraits showing (c) no basin instability for $r_2 = 2.82$, (d) marginal basin instability for $r_2 = 2.41$, and (e) partial basin instability of $\Gamma(r)$ on $\Delta_r$ for $r_2 = 2$. Basin stable parts of  $\Gamma(r)$ are shown in green, basin unstable parts of $\Gamma(r)$ are shown in red.
    The other parameter values are given in Appendix~\ref{app:maymodel}, Table~\ref{tab:parValues}.
    }
    \label{fig:BI_may}
\end{figure}

\begin{figure}[ht]
    \centering 
    \includegraphics[width=1\textwidth]{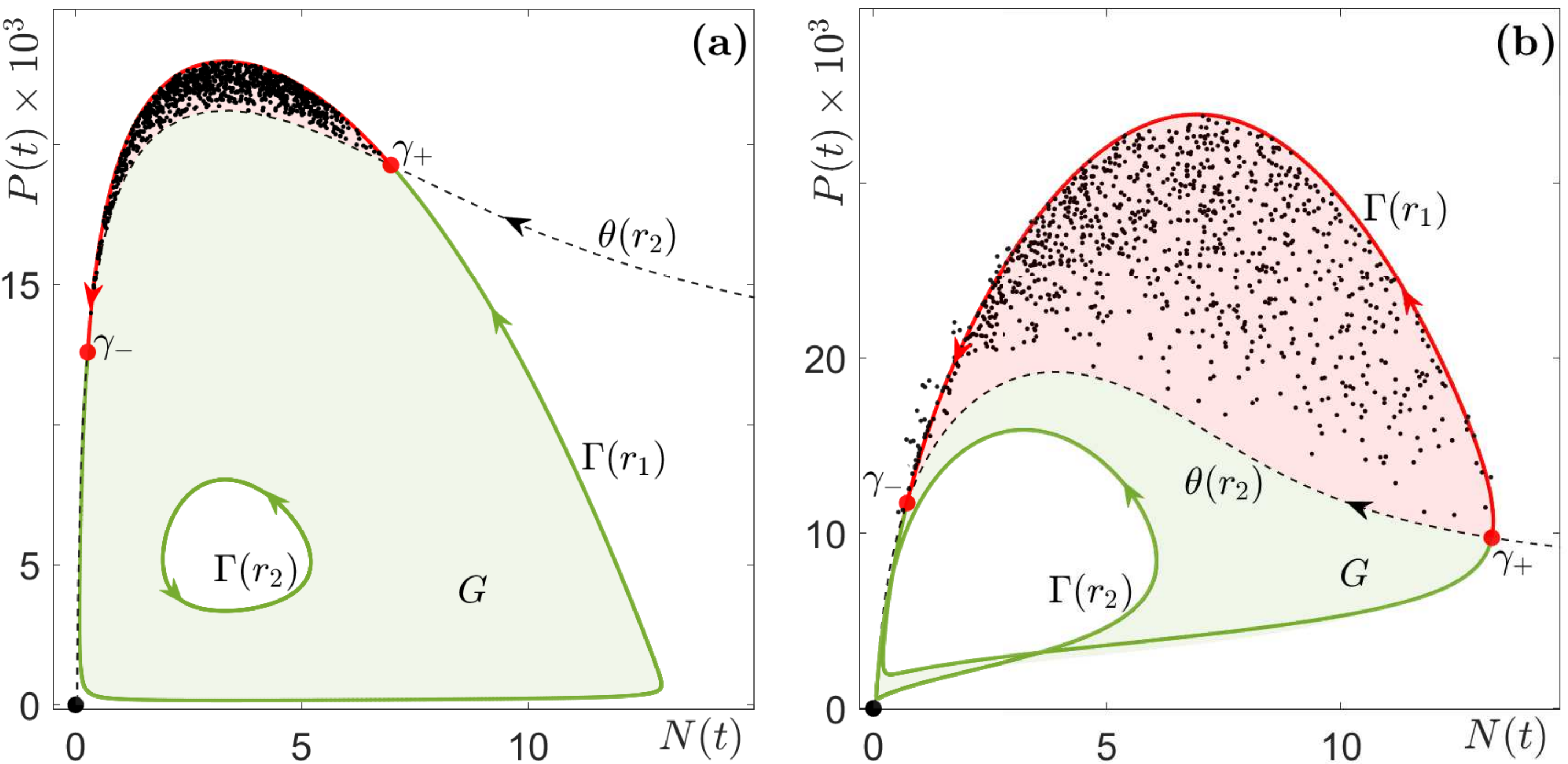}
    \caption{
     The concept of partial basin instability on a parameter path $\Delta_r$ with $r \in [r_2,r_1]$ (see Def.~\ref{def:partialBI_set}) is applied to the union $G$ of all predator-prey limit cycles $\Gamma(r)$ on the path (see Eq.\eqref{eq:G}) to explain the counter-intuitive P-tipping phenomenon uncovered in Figs.~\ref{fig:cli_dist_RM} and~\ref{fig:cli_dist_May}.
    The (black dots) states from which the system
    P-tips to extinction agree perfectly with the (pink) basin unstable parts of $G$ for
    (a) the RMA model~\eqref{eq:RM1} with $r_1=2.5$, $r_2=1.6$, and $\delta=2.2$, and 
    (b) the May model~\eqref{eq:May1}  with $r_1=3.3$, $r_2=2$, and $q=205$. The other parameter values are given in Appendix~\ref{app:maymodel}, Table~\ref{tab:parValues}.
   }
    \label{fig:BI_strip_may}
\end{figure}

\subsection{Partial basin instability of predator-prey cycles}

We now concentrate on the bistable regions labelled ``Oscillatory Coexistence or Extinction",
 apply Definitions~\ref{def:partialBI_set} and~\ref{def:BI} to predator-prey cycles, and show that
\begin{itemize}
    \item 
    Predator-prey cycles $\Gamma(r)$ can be partially basin unstable on suitably chosen parameter paths.
    \item Both predator-prey models have large parameter regions of partial basin instability. When superimposed onto classical bifurcation diagrams, these regions reveal P-tipping instabilities that cannot be captured by the classical autonomous bifurcation analysis.
     \item
    Partial basin instability of $\Gamma(r)$ in the frozen system is sufficient for the occurrence of P-tipping in the nonautonomous system.
\end{itemize}

The base attractor is the predator-prey limit cycle $\Gamma(r)$, and the alternative attractor is the extinction equilibrium $e_0$.
The basin boundary of $\Gamma(r)$ 
is the Allee threshold $\theta(r)$, which can be computed as the stable invariant manifold of the saddle equilibrium $e_s(r)$:
$$
\theta(r) := W^s(e_s(r)) = 
\left\{
(N_0,P_0)\in\mathbb{R}^2: 
\left(N(t),P(t)\right)\to e_s(r)\;\;\mbox{as}\;\; t\to +\infty
\right\}.
$$
In the RMA frozen model, $e_s(r)$ is the saddle Allee equilibrium $e_2$, whereas in the May frozen model,  $e_s(r)$ is the saddle coexistence equilibrium $e_4(r)$ that lies near the repelling Allee equilibrium $e_2$.
To uncover the full extent of partial basin instability for the predator-prey cycles $\Gamma(r)$, 
we fix a point $p_1$ that lies within the region
labelled ``Oscillatory Coexistence or Extinction"; see Figs.~\ref{fig:BI_RM}\bluea~and~\ref{fig:BI_may}\bluea. Then, we apply definition~\eqref{eq:BIR} to identify all points $p_2$ 
within this region
such that the predator-prey limit cycle $\Gamma(p_1)$ is not contained in the closure of the basin of attraction of $\Gamma(p_2)$.
The ensuing (light grey) regions of {\em partial basin instability} bounded by the (dark grey) curves of {\em marginal basin instability} are superimposed on the classical bifurcation diagrams in Figs.~\ref{fig:BI_RM}\bluea~and~\ref{fig:BI_may}\bluea. Note that the basin instability regions $BI(\Gamma,p_1)$ depend on the choice of $p_1$, and are labelled simply $BI$ for brevity.
To illustrate the underlying mechanism in the $(N,P)$ phase plane, {  we restrict} to parameter paths $\Delta_r$ that are straight horizontal lines from $p_1$ in the direction of decreasing $r$. In other words, we set $p=r$; see Figs.~\ref{fig:BI_RM}\bluea~ and~\ref{fig:BI_may}\bluea.
When $r_2\in\Delta_r$ lies on the dark grey curve of marginal basin instability, there is a single point of tangency between $\Gamma(r_1)$ and and $\theta(r_2)$, denoted $\gamma_\pm$ in Figs. \ref{fig:BI_RM}\blued\, and \ref{fig:BI_may}\blued.
When $r_2\in\Delta_r$ lies within the light grey region of partial basin instability, there are two points of intersection between $\Gamma(r_1)$ and and $\theta(r_2)$, denoted $\gamma_-$ and $\gamma_+$ in Figs.~\ref{fig:BI_RM}\bluee~and~\ref{fig:BI_may}\bluee. 
These two points bound the (red) part of the cycle that is basin unstable. The corresponding {\em basin unstable phases} are shown in Figs.~\ref{fig:BI_RM}\blueb~and~\ref{fig:BI_may}\blueb. 
Suppose that $r(t) = r_1$, and a trajectory of the nonautonomous system is on the same side of $\theta(r_2)$ as the (red) basin unstable part of $\Gamma(r_1)$. Then, when $r(t)$ changes from $r_1$ to $r_2$, the trajectory finds itself in the basin of attraction of the extinction equilibrium $e_0$, and will thus approach $e_0$.

The striking similarity is that predator-prey cycles from both models exhibit partial basin instability upon decreasing $r$. This decrease corresponds to climate-induced decline in the resources or in the quality of habitat.  Furthermore, while the predator-prey cycle in the May model has a noticeably wider range of basin unstable phases, neither cycle appears to be totally basin unstable. All these observations are consistent with the counter-intuitive properties (P1)--(P3) of P-tipping identified in the 
numerical experiments in Sec.~\ref{sec:cipst}.

\subsection{Partial basin instability explains P-tipping}

Now, we can demonstrate that partial basin instability of $\Gamma(r)$ in the autonomous predator-prey frozen systems explains and gives simple testable criteria for the occurrence of P-tipping in the nonautonomous systems. The families of attracting predator-prey limit cycles $\Gamma(r)$, and their basin boundaries $\theta(r)$, are the two crucial components of the discussion below.

First, recall the numerical P-tipping experiments from Sec.~\ref{sec:cipst}, 
and focus on the crescent shaped `clouds' of states from which P-tipping occurs; see the black dots Figs.~\ref{fig:cli_dist_RM}~and~\ref{fig:cli_dist_May}. Second, recognise that each P-tipping
event occurs for a different value of $r_\textrm{pre}\in[r_2,r_1]$, and thus from  a different predator-prey cycle $\Gamma(r_\textrm{pre})$ {  or its neighbourhood}.
Therefore, we must consider the union of all cycles from the family 
along the parameter path $\Delta_r$ bounded by $r_2$ and $r_1$:
{ 
\begin{equation}
\label{eq:G}
    G:= \big\{ \Gamma(r):\, r\in[r_2,r_1] \big\},
\end{equation}
}
which is shown in Fig.~\ref{fig:BI_strip_may}. Furthermore, we use the basin boundary $\theta(r_2)$ of the cycle $\Gamma(r_2)$ at the left end of the path to divide $G$ into its (light green) {\em basin stable part} and (pink) {\em basin unstable part} on $\Delta_r$
with $r\in [r_2,r_1]$.
The `clouds' of states from which P-tipping occurs agree perfectly with the basin unstable part of $G$. 
A few black dots that lie slightly outside the basin unstable part of $G$ in Fig.~\ref{fig:BI_strip_may}\blueb~correspond to those P-tipping events that occur from 
{  states that have not converged to the limit cycle $\Gamma(r_\textrm{pre})$ and lie visibly away from $\Gamma(r_\textrm{pre})$ when the switch that causes tipping happens. Those P-tipping events occur if the time interval $\ell$ during which $r(t)=r_\textrm{pre}$ is shorter than the time of convergence to the limit cycle $\Gamma(r_\textrm{pre})$ in the autonomous frozen system. For this particular parameter path, we could not detect any tipping events in the absence of partial basin instability of $\Gamma(r)$. 
However, we could detect multiple ``rescue events" described in Section~\ref{sec:basin_instability}\ref{sec:pbipt} (not shown in the figure). In a ``rescue event",
the system leaves the basin of attraction of the predator-pray 
cycle after a switch that gives basin instability, but avoids tipping upon this switch because it re-enters the basin of attraction of the predator-prey cycle after some future switch.
 ``Rescue events" occur if the time interval $\ell$ during which $r(t)=r_\textrm{pre}$ is shorter than the time of convergence to the extinction equilibrium $e_0$ in the autonomous frozen system. In summary, the general concept of partial basin instability of $\Gamma(r)$ on a parameter path $\Delta_r$ from  Definition~\ref{def:partialBI_set} is an excellent indicator for the occurrence of P-tipping in the RMA~\eqref{eq:RM1}  and May~\eqref{eq:May1} models.
}

\section{Conclusions}
\label{sec:conclusions}

This paper studies nonlinear tipping phenomena, or critical transitions, in nonautonomous dynamical systems with time-varying
external inputs. In addition to the well-known critical factors for tipping in systems that are stationary in the absence of external inputs, {  namely bifurcation, rate of change, and noise,} we identify here the
{\em phase} {  of predator-prey limit cycles and nearby  oscillations} as a new critical factor in systems that are cyclic in the absence of external inputs. 

To illustrate the new tipping phenomenon in a realistic setting, we consider two
paradigmatic predator-prey models with an Allee effect, namely the Rosenzweig-MacArthur model~\cite{rosenzweig1963} and the May model~\cite{May2019}. 
We describe temporal changes in the carrying capacity of the ecosystem
with real climate variability records from different communities in the boreal and deciduous-boreal forest~\cite{Marley2020}, and 
use realistic parameter values for the Canada lynx and snowshoe hare system~\cite{tyson:2009, strohm:2009}.
Monte Carlo simulation reveals a robust phenomenon, where a drop in the carrying capacity 
tips the ecosystem from a predator-prey {  oscillations}  to extinction. The special and somewhat counter-intuitive result is that tipping occurs: (i) without crossing any bifurcations, and (ii) only from certain phases of the {  oscillations}. Thus, we refer to this phenomenon as {\em {  phase} tipping} {  ({\em partial tipping})}, or simply {\em P-tipping}.  Intuitively, P-tipping from predator-prey {  oscillations} to extinction arises because
a fixed drop in prey resources has distinctively different effects when applied during the phases of the {  oscillations} with the fastest growth and the fastest decline of  prey.

Motivated by the outcome of the simulation, we develop an accessible and general mathematical framework to analyse P-tipping and reveal the underlying dynamical mechanism. Specifically, we employ notions from set-valued dynamics to extend the geometric concept of basin instability, introduced 
in~\cite{o2019} for equilibria, to limit cycles. The main idea is to consider the autonomous frozen system  with different but fixed-in-time values of the external input along some parameter path, and examine the position of the limit cycle at some point on the path relative to the position of its basin of attraction at other points on the path. 
First, we define different types of basin instability for limit cycles, and focus on {\em partial basin instability} that does not exist for equilibria. Second, we show that  partial basin instability in the autonomous frozen
system is {  necessary and} sufficient for  the occurrence of P-tipping in the nonautonomous system with a {  single-switch} external input. {  Furthermore, we discuss applicability of this result to multi-switch external inputs.} Third, we relate our results  to those of Alkhayuon and Ashwin~\cite{alkhayuon2018} on rate-induced tipping from limit cycles. 

We then apply the general framework to the ecosystem models and explain the counter-intuitive transitions from {  certain phases of} predator-prey {  oscillations} to extinction.
We use classical autonomous bifurcation analysis to identify parameter regions with bistability between predator-prey cycles and extinction. In this way, we show that predator-prey cycles can be partially basin unstable on typical parameter paths  within these bistability regions. Moreover, we superimpose regions of partial basin instability onto classical autonomous bifurcation diagrams to reveal P-tipping instabilities that are robust but cannot be captured by classical bifurcation analysis. 

We believe that this approach will enable scientists to uncover P-tipping in many different cyclic systems from applications, ranging from natural science  and engineering to economics.
For example, the predator-prey paradigm is found across biological applications modelling, including epidemiology \cite{eilersen:2020}, pest control \cite{buxton:2020}, fisheries \cite{pujaru:2020}, cancer \cite{griffiths:2020, georgiou:2020}, and agriculture \cite{murano:2019, fort:2017}.  The fundamental relationship described in predator-prey models also appears in many areas outside of the biological sciences, with recent examples including atmospheric sciences \cite{lunderman:2020}, economic development \cite{mehlum2003, mehlum2006}, trade and financial crises~\cite{mesly:2020, edwards:2020,huck:2020}, and land management \cite{jenkins:2020}.
External disturbances of different kinds exist in all of these systems,   suggesting that the P-tipping behaviours discovered in this paper are of broad practical relevance.

Furthermore, the concept of P-tipping, for base states that are attracting limit cycles with regular basin boundaries, naturally extends to more complicated base states, such as quasiperiodic tori and chaotic attractors, and to irregular (e.g. fractal) basin boundaries~\cite{kaszas2019,alkhayuon2020,ashwin2021,lohmann2021}. 
Defining phase for 
more complicated cycles in higher dimensions, and for non-periodic 
{  oscillations}, will usually require a different approach. For example, 
{
one could define phase for an attracting limit cycle in a multidimensional system in terms of its period $T$ as a linear function of time $\varphi= 2\pi t/T$. This definition is independent of the coordinate system, and can be extended to every point in the basin of attraction using isochrones.
Another approach is to} work with a time series of a single observable and use the Hilbert transform to construct the complex-valued analytic signal, and then extract the so-called {\em instantaneous phase}~\cite{bracewell1986,rosenblum1996}. This phase variable may provide valuable physical insights into the problem of P-tipping when the polar coordinate approach does not work, or when the base attractor or its basin boundary have complicated geometry and are difficult to visualise. 
Such systems will likely exhibit even more counter-intuitive tipping behaviours,
but their analysis requires
mathematical techniques  beyond the scope of this paper.

Another interesting research question is that of early warning indicators for P-tipping.
In the past decade, many studies of noisy real-world time-series records revealed prompt changes in the statistical properties of the data prior to tipping~\cite{scheffer2009,ditlevsen2010,lenton2011,ritchie2016}, which appear to be generic for tipping from equilibria. 
However, it is  unclear if these statistical early warning indicators appear for P-tipping, or if one needs to identify
alternatives such as Finite Time Lyapunov Exponent (FTLE)~\cite{remo2019}.


\section*{Data Accessibility and Authors’ Contributions}
The codes used to conduct simulations and  generate figures are available via the GitHub repository \cite{alkhayuon2021}. All authors contributed to the numerical computations and to the writing of the manuscript. HA and SW contributed to the theoretical results in Sec.~\ref{sec:basin_instability}

\section*{Funding and Acknowledgements}
HA and SW are funded by Enterprise Ireland grant No. 20190771. RCT is funded by NSERC Discovery Grant RGPIN-2016-05277. We would like to thank Johan Dubbeldam, Cris Hasan, Bernd Krauskopf, Jessa Marley,  Emma McIvor and three anonymous reviewers for their constructive and insightful comments on this work, including an alternative phase definition in terms of isochrones.


\bibliographystyle{unsrt}
\bibliography{biblio}

\begin{thebibliography}{10}

\bibitem{scheffer2009}
Marten Scheffer, Jordi Bascompte, William~A Brock, Victor Brovkin, Stephen~R
  Carpenter, Vasilis Dakos, Hermann Held, Egbert~H Van~Nes, Max Rietkerk, and
  George Sugihara.
\newblock Early-warning signals for critical transitions.
\newblock {\em Nature}, 461(7260):53--59, 2009.

\bibitem{lenton2008}
Timothy~M Lenton, Hermann Held, Elmar Kriegler, Jim~W Hall, Wolfgang Lucht,
  Stefan Rahmstorf, and Hans~Joachim Schellnhuber.
\newblock Tipping elements in the earth's climate system.
\newblock {\em Proceedings of the national Academy of Sciences},
  105(6):1786--1793, 2008.

\bibitem{ashwin2012}
Peter Ashwin, Sebastian Wieczorek, Renato Vitolo, and Peter Cox.
\newblock Tipping points in open systems: bifurcation, noise-induced and
  rate-dependent examples in the climate system.
\newblock {\em Philosophical Transactions of the Royal Society A: Mathematical,
  Physical and Engineering Sciences}, 370(1962):1166--1184, 2012.

\bibitem{thompson1994}
J~Michael~T Thompson, HB~Stewart, and Y~Ueda.
\newblock Safe, explosive, and dangerous bifurcations in dissipative dynamical
  systems.
\newblock {\em Physical Review E}, 49(2):1019, 1994.

\bibitem{thompson2011}
J~Michael~T Thompson and Jan Sieber.
\newblock Climate tipping as a noisy bifurcation: a predictive technique.
\newblock {\em IMA Journal of Applied Mathematics}, 76(1):27--46, 2011.

\bibitem{kuehn2011}
Christian Kuehn.
\newblock A mathematical framework for critical transitions: Bifurcations,
  fast--slow systems and stochastic dynamics.
\newblock {\em Physica D: Nonlinear Phenomena}, 240(12):1020--1035, 2011.

\bibitem{wieczorek2011}
Sebastian Wieczorek, Peter Ashwin, Catherine~M Luke, and Peter~M Cox.
\newblock Excitability in ramped systems: the compost-bomb instability.
\newblock {\em Proceedings of the Royal Society A: Mathematical, Physical and
  Engineering Sciences}, 467(2129):1243--1269, 2011.

\bibitem{perryman2014}
Clare Perryman and Sebastian Wieczorek.
\newblock Adapting to a changing environment: non-obvious thresholds in
  multi-scale systems.
\newblock {\em Proceedings of the Royal Society A: Mathematical, Physical and
  Engineering Sciences}, 470(2170):20140226, 2014.

\bibitem{vanselow2019}
Anna Vanselow, Sebastian Wieczorek, and Ulrike Feudel.
\newblock When very slow is too fast-collapse of a predator-prey system.
\newblock {\em Journal of theoretical biology}, 479:64--72, 2019.

\bibitem{wieczorek2020}
Sebastian Wieczorek, Chun Xie, and Peter Ashwin.
\newblock Rate-induced tipping: Thresholds, edge states and connecting orbits.
\newblock {\em in preparation}, 2021.

\bibitem{kuehn2020}
Christian Kuehn and Iacopo~P Longo.
\newblock Estimating rate-induced tipping via asymptotic series and a
  melnikov-like method.
\newblock {\em arXiv preprint arXiv:2011.04031}, 2020.

\bibitem{scheffer2008}
Marten Scheffer, Egbert~H Van~Nes, Milena Holmgren, and Terry Hughes.
\newblock Pulse-driven loss of top-down control: the critical-rate hypothesis.
\newblock {\em Ecosystems}, 11(2):226--237, 2008.

\bibitem{o2019}
Paul~E O'Keeffe and Sebastian Wieczorek.
\newblock Tipping phenomena and points of no return in ecosystems: Beyond
  classical bifurcations.
\newblock {\em SIAM Journal on Applied Dynamical Systems}, 19(4):2371--2402,
  2020.

\bibitem{ashwin2017}
Peter Ashwin, Clare Perryman, and Sebastian Wieczorek.
\newblock Parameter shifts for nonautonomous systems in low dimension:
  bifurcation-and rate-induced tipping.
\newblock {\em Nonlinearity}, 30(6):2185, 2017.

\bibitem{alkhayuon2019}
Hassan Alkhayuon, Peter Ashwin, Laura~C Jackson, Courtney Quinn, and Richard~A
  Wood.
\newblock Basin bifurcations, oscillatory instability and rate-induced
  thresholds for atlantic meridional overturning circulation in a global
  oceanic box model.
\newblock {\em Proceedings of the Royal Society A}, 475(2225):20190051, 2019.

\bibitem{Hartl2019}
Michael Hartl.
\newblock {\em Non-autonomous random dynamical systems: Stochastic
  approximation and rate-induced tipping}.
\newblock PhD thesis, Imperial College London, 2019.

\bibitem{kiers2020}
Claire Kiers and Christopher~KRT Jones.
\newblock On conditions for rate-induced tipping in multi-dimensional dynamical
  systems.
\newblock {\em Journal of Dynamics and Differential Equations}, 32(1):483--503,
  2020.

\bibitem{halekotte2020}
Lukas Halekotte and Ulrike Feudel.
\newblock Minimal fatal shocks in multistable complex networks.
\newblock {\em Scientific reports}, 10(1):1--13, 2020.

\bibitem{franovic2018}
Igor Franovi{\'c}, Oleh~E Omel'chenko, and Matthias Wolfrum.
\newblock Phase-sensitive excitability of a limit cycle.
\newblock {\em Chaos: An Interdisciplinary Journal of Nonlinear Science},
  28(7):071105, 2018.

\bibitem{benzi1981}
Roberto Benzi, Alfonso Sutera, and Angelo Vulpiani.
\newblock The mechanism of stochastic resonance.
\newblock {\em Journal of Physics A: mathematical and general}, 14(11):L453,
  1981.

\bibitem{ditlevsen2010}
Peter~D Ditlevsen and Sigfus~J Johnsen.
\newblock Tipping points: early warning and wishful thinking.
\newblock {\em Geophysical Research Letters}, 37(19), 2010.

\bibitem{ritchie2016}
Paul Ritchie and Jan Sieber.
\newblock Early-warning indicators for rate-induced tipping.
\newblock {\em Chaos: An Interdisciplinary Journal of Nonlinear Science},
  26(9):093116, 2016.

\bibitem{chen2019}
Yuxin Chen, John~A Gemmer, Mary Silber, and Alexandria Volkening.
\newblock Noise-induced tipping under periodic forcing: Preferred tipping phase
  in a non-adiabatic forcing regime.
\newblock {\em Chaos: An Interdisciplinary Journal of Nonlinear Science},
  29(4):043119, 2019.

\bibitem{freund2006}
Jan~A Freund, Sebastian Mieruch, Bettina Scholze, Karen Wiltshire, and Ulrike
  Feudel.
\newblock Bloom dynamics in a seasonally forced phytoplankton--zooplankton
  model: trigger mechanisms and timing effects.
\newblock {\em Ecological complexity}, 3(2):129--139, 2006.

\bibitem{medeiros2017}
Everton~S Medeiros, Iber{\^e}~L Caldas, Murilo~S Baptista, and Ulrike Feudel.
\newblock Trapping phenomenon attenuates the consequences of tipping points for
  limit cycles.
\newblock {\em Scientific reports}, 7:42351, 2017.

\bibitem{alkhayuon2018}
Hassan Alkhayuon and Peter Ashwin.
\newblock Rate-induced tipping from periodic attractors: Partial tipping and
  connecting orbits.
\newblock {\em Chaos: An Interdisciplinary Journal of Nonlinear Science},
  28(3):033608, 2018.

\bibitem{bathiany2018}
S~Bathiany, M~Scheffer, EH~Van~Nes, MS~Williamson, and TM~Lenton.
\newblock Abrupt climate change in an oscillating world.
\newblock {\em Scientific reports}, 8(1):1--12, 2018.

\bibitem{kaszas2019}
B{\'a}lint Kasz{\'a}s, Ulrike Feudel, and Tam{\'a}s T{\'e}l.
\newblock Tipping phenomena in typical dynamical systems subjected to parameter
  drift.
\newblock {\em Scientific reports}, 9(1):1--12, 2019.

\bibitem{keane2020}
Andrew Keane, Bernd Krauskopf, and Timothy~M Lenton.
\newblock Signatures consistent with multifrequency tipping in the atlantic
  meridional overturning circulation.
\newblock {\em Physical Review Letters}, 125(22):228701, 2020.

\bibitem{longo2020}
Iacopo~P Longo, Carmen N{\'u}nez, Rafael Obaya, and Martin Rasmussen.
\newblock Rate-induced tipping and saddle-node bifurcation for quadratic
  differential equations with nonautonomous asymptotic dynamics.
\newblock {\em SIAM Journal on Applied Dynamical Systems}, 20(1):500--540,
  2021.

\bibitem{lohmann2021}
Johannes Lohmann and Peter~D Ditlevsen.
\newblock Risk of tipping the overturning circulation due to increasing rates
  of ice melt.
\newblock {\em Proceedings of the National Academy of Sciences}, 118(9), 2021.

\bibitem{ashwin2021}
Peter Ashwin and Julian Newman.
\newblock Physical invariant measures and tipping probabilities for chaotic
  attractors of asymptotically autonomous systems.
\newblock {\em The European Physical Journal Special Topics}, pages 1--14,
  2021.

\bibitem{anishchenko1993}
VS~Anishchenko, AB~Neiman, and MA~Safanova.
\newblock Stochastic resonance in chaotic systems.
\newblock {\em Journal of statistical physics}, 70(1-2):183--196, 1993.

\bibitem{gammaitoni1998}
Luca Gammaitoni, Peter H{\"a}nggi, Peter Jung, and Fabio Marchesoni.
\newblock Stochastic resonance.
\newblock {\em Reviews of modern physics}, 70(1):223, 1998.

\bibitem{medeiros2019}
Everton~S Medeiros, Rene~O Medrano-T, Iber{\^e}~L Caldas, Tam{\'a}s T{\'e}l,
  and Ulrike Feudel.
\newblock State-dependent vulnerability of synchronization.
\newblock {\em Physical Review E}, 100(5):052201, 2019.

\bibitem{rockstrom:2009}
Johan Rockstr{\"o}m, Will Steffen, Kevin Noone, {\AA}sa Persson, F~Stuart
  Chapin, Eric~F Lambin, Timothy~M Lenton, Marten Scheffer, Carl Folke,
  Hans~Joachim Schellnhuber, et~al.
\newblock A safe operating space for humanity.
\newblock {\em nature}, 461(7263):472--475, 2009.

\bibitem{barnosky:2012}
A.D. Barnosky, E.A. Hadly, J~Bascompte, E.L. Berlow, J.H. Brown, M.~Fortellus,
  W.M Getz, J.~Hare, A.~Hastings, P.A. Marquet, and et~al.
\newblock Approaching a state shift in earth's biosphere.
\newblock {\em Nature}, 486:52--58, 2012.

\bibitem{rosenzweig1963}
Michael~L Rosenzweig and Robert~H MacArthur.
\newblock Graphical representation and stability conditions of predator-prey
  interactions.
\newblock {\em The American Naturalist}, 97(895):209--223, 1963.

\bibitem{May2019}
Robert~M May.
\newblock {\em Stability and complexity in model ecosystems}, volume~1.
\newblock Princeton university press, 2019.

\bibitem{vitense:2016}
K.~Vitense, A.J. Wirsing, R.C. Tyson, and J.J. Anderson.
\newblock Theoretical impacts of habitat loss and generalist predation on
  predator-prey cycles.
\newblock {\em Ecological Modelling}, 327:85--94, 2016.

\bibitem{sauve:2020}
A.M. Sauve, R.A. Taylor, and F.~Barraquand.
\newblock The effect of seasonal strength and abruptness on predaor-prey
  dynamics.
\newblock {\em Journal of Theoretical Biology}, 491(110175), 2020.

\bibitem{sarker:2020}
S.~Sarker, A.K. Yadav, M.~Akter, M.S. Hossain, S.R. Chowdhury, M.A. Kabir, and
  S.M. Sharifuzzaman.
\newblock Rising temperature and marine plankton community dynamics: Is warming
  bad?
\newblock {\em Ecological Complexity}, 43(100857), 2020.

\bibitem{tyson:2009}
R.~Tyson, S.~Haines, and K.E. Hodges.
\newblock Modelling the {C}anada lynx and snowshoe hare population cycle: the
  role of specialist predators.
\newblock {\em Theoretical Ecology}, 3:97--111, 2009.

\bibitem{strohm:2009}
S.~Strohm and R.~Tyson.
\newblock The effect of habitat fragmentation on cyclic population dynamics:
  {A} numerical study.
\newblock {\em Bulletin of Mathematical Biology}, 71:1323--1348, 2009.

\bibitem{Marley2020}
Hassan Alkhayuon, Jessa Marley, Sebastian Wieczorek, and Tyson~Rebecca C.
\newblock Contemporary climate variability and cyclic populations: Implications
  for persistence.
\newblock {\em unpublished preprint}, 2021.

\bibitem{kendall1998}
Bruce~E Kendall, John Prendergast, and Ottar~N Bjornstad.
\newblock The macroecology of population dynamics: taxonomic and biogeographic
  patterns in population cycles.
\newblock {\em Ecology letters}, 1(3):160--164, 1998.

\bibitem{turchin:2003}
P.~Turchin.
\newblock {\em Complex population dynamics}.
\newblock Princeton University Press, Princeton, N.J., 2003.

\bibitem{stewart:1993}
A.J. Stewart.
\newblock {\em Graphs of recorded forest insect outbreaks up to 1992 in the
  Nelson Forest Region}.
\newblock Victoria Forestry Canada, Pacific \& Yukon region, Pacific Forestry
  Centre, 1993.

\bibitem{krebs:2001}
C.J. Krebs, S.~Boutin, and R.~Boonstra, editors.
\newblock {\em Ecosystem dynamics of the boreal forest: the {K}luane project}.
\newblock Oxford University Press, New York, 2001.

\bibitem{vanderbolt:2018}
B.~{van der Bolt}, E.H. {van Nies}, S.~Bathiany, M.E. Vollegregt, and
  M.~Scheffer.
\newblock Climate reddening increases the chance of critical transitions.
\newblock {\em Nature Climate Change}, 8:478--484, 2018.

\bibitem{richardson:2005}
J.~Richardson and C.~Robinson.
\newblock Climate change and marine plankton.
\newblock {\em Trends in Ecology and Evolution}, 20(6):337--344, 2005.

\bibitem{harley:2011}
C.D.G. Harley.
\newblock Climate change, keystone predations, and biodiversity loss.
\newblock {\em Science}, 334(6059):1124--1127, 2011.

\bibitem{ummenhofer:2016}
C.C. Ummenhofer and G.A. Meehl.
\newblock Extreme weather and climate events with ecological relevance: a
  review.
\newblock {\em Philosophical Transactions of the Royal Society B},
  372(20160135), 2016.

\bibitem{pauli:2003}
H.~Pauli, G.~Gottfried, and G.~Grabherr.
\newblock Effects of climate change on the alpine and nival vegetation of the
  {A}lps.
\newblock {\em Journal of Mountain Ecology}, 7 (Suppl), 2003.

\bibitem{zimmerman:2020}
S.R.H. Zimmerman and D.B. Wahl.
\newblock Holocene paleoclimate change in the western {US: T}he importance of
  chronology in discerning patterns and drivers.
\newblock {\em Quaternary Science Reviews}, 246(106487), 2020.

\bibitem{elkenawy:2020}
Ahmed~M El~Kenawy, Ali Al~Buloshi, Talal Al-Awadhi, Noura Al~Nasiri, Francisco
  Navarro-Serrano, Salim Alhatrushi, SM~Robaa, Fernando Dom{\'\i}nguez-Castro,
  Matthew~F McCabe, Petra-Manuela Schuwerack, et~al.
\newblock Evidence for intensification of meteorological droughts in oman over
  the past four decades.
\newblock {\em Atmospheric Research}, 246:105126, 2020.

\bibitem{spinoni:2018}
J.~Spinoni, J.V. Vogt, G.~Naumann, P.~Barbosa, and A.~Dosio.
\newblock Will drought events become more frequent and severe in {Europe?}
\newblock {\em International Journal of Climatology}, 38(4):1718--1736, 2018.

\bibitem{smale:2019}
Dan~A. Smale, Thomas Wernberg, Eric C.~J. Oliver, Mads Thomsen, Ben~P. Harvey,
  Sandra~C. Straub, Michael~T. Burrows, Lisa Alexander, V, Jessica~A.
  Benthuysen, Markus~G. Donat, Ming Feng, Alistair~J. Hobday, Neil~J. Holbrook,
  Sarah~E. Perkins-Kirkpatrick, Hillary~A. Scannell, Alex Sen~Gupta, Ben~L.
  Payne, and Pippa~J. Moore.
\newblock {Marine heatwaves threaten global biodiversity and the provision of
  ecosystem services}.
\newblock {\em {Nature Climate Change}}, {9}({4}):{306+}, {APR} {2019}.

\bibitem{veh:2020}
G.~Veh, O.~Korup, and A.~Walz.
\newblock Hazard from {Himalayan} glacier lake outburst floods.
\newblock {\em Proceedings of the National Academy of Sciences},
  117(2):907--912, 2020.

\bibitem{bloeschl:2019}
G.~Bloeschl, J.~Hall, A.~Viglione, R.~A.~P. Perdigao, J.~Parajka, B.~Merz,
  D.~Lun, B.~Arheimer, Giuseppe~T. Aronica, Ardian Bilibashi, and et~al.
\newblock {Changing climate both increases and decreases European river
  floods}.
\newblock {\em {Nature}}, {573}({7772}):{108+}, {SEP 5} {2019}.

\bibitem{depietri:2018}
Yaella Depietri and Timon McPhearson.
\newblock {Changing urban risk: 140 years of climatic hazards in New York
  City}.
\newblock {\em {Climatic Change}}, {148}({1-2}):{95--108}, {MAY} {2018}.

\bibitem{knutson:2020}
Thomas Knutson, Suzana~J. Camargo, Johnny C.~L. Chan, Kerry Emanuel, Chang-Hoi
  Ho, James Kossin, Mrutyunjay Mohapatra, Masaki Satoh, Masato Sugi, Kevin
  Walsh, and Liguang Wu.
\newblock Tropical cyclones and climate change assessment: {Part II}:
  {P}rojected response to anthropogenic warming.
\newblock {\em {Bulletin of the American Meteorological Society}},
  {101}({3}):{E303--E322}, {MAR} {2020}.

\bibitem{chauvin:2020}
Fabrice Chauvin, Romain Pilon, Philippe Palany, and Ali Belmadani.
\newblock Future changes in atlantic hurricanes with the rotated-stretched
  arpege-climat at very high resolution.
\newblock {\em Climate Dynamics}, 54(1):947--972, 2020.

\bibitem{mehlum2003}
Halvor Mehlum, Karl Moene, and Ragnar Torvik.
\newblock Predator or prey?: Parasitic enterprises in economic development.
\newblock {\em European Economic Review}, 47(2):275--294, 2003.

\bibitem{mehlum2006}
Halvor Mehlum, Karl Moene, and Ragnar Torvik.
\newblock Cursed by resources or institutions?
\newblock {\em World Economy}, 29(8):1117--1131, 2006.

\bibitem{mirza2019}
M~Usman Mirza, Andries Richter, Egbert~H van Nes, and Marten Scheffer.
\newblock Technology driven inequality leads to poverty and resource depletion.
\newblock {\em Ecological Economics}, 160:215--226, 2019.

\bibitem{Berryman1992}
Alan~A Berryman.
\newblock The orgins and evolution of predator-prey theory.
\newblock {\em Ecology}, 73(5):1530--1535, 1992.

\bibitem{roy2007}
Shovonlal Roy and J~Chattopadhyay.
\newblock The stability of ecosystems: a brief overview of the paradox of
  enrichment.
\newblock {\em Journal of biosciences}, 32(2):421--428, 2007.

\bibitem{picoche:2019}
Coralie Picoche and Frederic Barraquand.
\newblock {How self-regulation, the storage effect, and their interaction
  contribute to coexistence in stochastic and seasonal environments}.
\newblock {\em {Theoretical Ecology}}, {12}({4}):{489--500}, {2019}.

\bibitem{ghosh:2019}
Partha Ghosh, Pritha Das, and Debasis Mukherjee.
\newblock {Persistence and Stability of a Seasonally Perturbed Three Species
  Stochastic Model of Salmonoid Aquaculture}.
\newblock {\em {Differential Equations and Dynamical Systems}},
  {27}({4}):{449--465}, {2019}.

\bibitem{bessho:2010}
Kazuhiro Bessho and Yoh Iwasa.
\newblock {Optimal seasonal schedules and the relative dominance of
  heteromorphic and isomorphic life cycles in macroalgae}.
\newblock {\em {Journal of Theoretical Biology}}, {267}({2}):{201--212},
  {2010}.

\bibitem{hanski:1995}
I~Hanski and E~Korpimaki.
\newblock {Microtine rodent dynamics in northern Europe - Parameterized odels
  for the predator-prey interaction}.
\newblock {\em {Ecology}}, {76}({3}):{840--850}, {1995}.

\bibitem{schoenmakers2021}
Sarah Schoenmakers and Ulrike Feudel.
\newblock A resilience concept based on system functioning: A dynamical systems
  perspective.
\newblock {\em Chaos: An Interdisciplinary Journal of Nonlinear Science},
  31(5):053126, 2021.

\bibitem{wilmers2007}
Christopher~C Wilmers, Eric Post, and Alan Hastings.
\newblock A perfect storm: the combined effects on population fluctuations of
  autocorrelated environmental noise, age structure, and density dependence.
\newblock {\em The American Naturalist}, 169(5):673--683, 2007.

\bibitem{Devroye2006}
Luc Devroye.
\newblock Nonuniform random variate generation.
\newblock {\em Handbooks in operations research and management science},
  13:83--121, 2006.

\bibitem{halmos1947}
Paul~R Halmos.
\newblock Invariant measures.
\newblock {\em Annals of Mathematics}, pages 735--754, 1947.

\bibitem{menck2013}
Peter~J Menck, Jobst Heitzig, Norbert Marwan, and J{\"u}rgen Kurths.
\newblock How basin stability complements the linear-stability paradigm.
\newblock {\em Nature physics}, 9(2):89--92, 2013.

\bibitem{ermentrout2002}
Bard Ermentrout.
\newblock {\em Simulating, analyzing, and animating dynamical systems: a guide
  to XPPAUT for researchers and students}.
\newblock SIAM, 2002.

\bibitem{kuznetsov2013}
Yuri~A Kuznetsov.
\newblock {\em Elements of applied bifurcation theory}, volume 112.
\newblock Springer Science \& Business Media, 2013.

\bibitem{eilersen:2020}
Andreas Eilersen and Kim Sneppen.
\newblock {The uneasy coexistence of predators and pathogens}.
\newblock {\em {European Physical Journal E}}, {43}({7}), {2020}.

\bibitem{buxton:2020}
Mmabaledi Buxton, Ross~N. Cuthbert, Tatenda Dalu, Casper Nyamukondiwa, and
  Ryan~J. Wasserman.
\newblock {Predator density modifies mosquito regulation in increasingly
  complex environments}.
\newblock {\em {Pest Management Science}}, {76}({6}):{2079--2086}, {2020}.

\bibitem{pujaru:2020}
Kanisha Pujaru and Tapan~Kumar Kar.
\newblock {Impacts of predator-prey interaction on managing maximum sustainable
  yield and resilience}.
\newblock {\em {Nonlinear Analysis-Modelling and Control}},
  {25}({3}):{400--416}, {2020}.

\bibitem{griffiths:2020}
Jason Griffiths, I, Pierre Wallet, Lance~T. Pflieger, David Stenehjem, Xuan
  Liu, Patrick~A. Cosgrove, Neena~A. Leggett, Jasmine~A. McQuerry, Gajendra
  Shrestha, Maura Rossetti, Gemalene Sunga, Philip~J. Moos, Frederick~R. Adler,
  Jeffrey~T. Chang, Sunil Sharma, and Andrea~H. Bild.
\newblock {Circulating immune cell phenotype dynamics reflect the strength of
  tumor-immune cell interactions in patients during immunotherapy}.
\newblock {\em {Proceedings of the National Academy of Sciences}},
  {117}({27}):{16072--16082}, {2020}.

\bibitem{georgiou:2020}
F.~Georgiou and N.~Thamwattana.
\newblock {Modelling phagocytosis based on cell-cell adhesion and prey-predator
  relationship}.
\newblock {\em {Mathematics and Computers in Simulation}},
  {171}({SI}):{52--64}, {2020}.

\bibitem{murano:2019}
Chie Murano, Satoe Kasahara, Seiya Kudo, Aya Inada, Sho Sato, Kana Watanabe,
  and Nobuyuki Azuma.
\newblock {Effectiveness of vole control by owls in apple orchards}.
\newblock {\em {Journal of Applied Ecology}}, {56}({3}):{677--687}, {2019}.

\bibitem{fort:2017}
Hugo Fort, Francisco Dieguez, Virginia Halty, and Juan~Manuel Soares~Lima.
\newblock {Two examples of application of ecological modeling to agricultural
  production: Extensive livestock farming and overyielding in grassland
  mixtures}.
\newblock {\em {Ecological Modelling}}, {357}:{23--34}, {2017}.

\bibitem{lunderman:2020}
Spencer Lunderman, Matthias Morzfeld, Franziska Glassmeier, and Graham
  Feingold.
\newblock {Estimating parameters of the nonlinear cloud and rain equation from
  a large-eddy simulation}.
\newblock {\em {Physica D-Nonlinear Phenomena}}, {410}, {2020}.

\bibitem{mesly:2020}
Olivier Mesly, David~W. Shanafelt, Nicolas Huck, and Francois-Eric Racicot.
\newblock {From wheel of fortune to wheel of misfortune: Financial crises,
  cycles, and consumer predation}.
\newblock {\em {Journal of Consumer Affairs}}, {54}({4}):{1195--1212}, {2020}.

\bibitem{edwards:2020}
Eric~C. Edwards, Dong-Hun Go, and Reza Oladi.
\newblock {Predator-prey dynamics in general equilibrium and the role of
  trade}.
\newblock {\em {Resource and Energy Economics}}, {61}, {2020}.

\bibitem{huck:2020}
Nicolas Huck, Hareesh Mavoori, and Olivier Mesly.
\newblock {The rationality of irrationality in times of financial crises}.
\newblock {\em {Economic Modelling}}, {89}:{337--350}, {2020}.

\bibitem{jenkins:2020}
David~G. Jenkins, Helmut Haberl, Karl-Heinz Erb, and Andrew~L. Nevai.
\newblock {Global human ``predation{''} on plant growth and biomass}.
\newblock {\em {Global Ecology and Biogeography}}, {29}({6}):{1052--1064},
  {2020}.

\bibitem{alkhayuon2020}
Hassan Alkhayuon and Peter Ashwin.
\newblock Weak tracking in nonautonomous chaotic systems.
\newblock {\em Physical Review E}, 102(5):052210, 2020.

\bibitem{bracewell1986}
Ronald~Newbold Bracewell and Ronald~N Bracewell.
\newblock {\em The Fourier transform and its applications}, volume 31999.
\newblock McGraw-Hill New York, 1986.

\bibitem{rosenblum1996}
Michael~G Rosenblum, Arkady~S Pikovsky, and J{\"u}rgen Kurths.
\newblock Phase synchronization of chaotic oscillators.
\newblock {\em Physical review letters}, 76(11):1804, 1996.

\bibitem{lenton2011}
Timothy~M Lenton.
\newblock Early warning of climate tipping points.
\newblock {\em Nature Climate Change}, 1(4):201--209, 2011.

\bibitem{remo2019}
Flavia Remo, Gabriel Fuhrmann, and Tobias J{\"a}ger.
\newblock On the effect of forcing of fold bifurcations and early-warning
  signals in population dynamics.
\newblock {\em arXiv preprint arXiv:1904.06507}, 2019.

\bibitem{alkhayuon2021}
Hassan Alkhayuon, Rebecca~C. Tyson, and Sebastian Wieczorek.
\newblock Phase-sensitive tipping: cyclic ecosystems subject to contemporary
  climate.
\newblock {\em GitHub Repository,}
  github.com/hassanalkhayuon/phaseSensitiveTipping, 2021.

\end{thebibliography}


\appendix
\section{Equilibria and bifurcations of the May frozen system}
\label{app:maymodel}

\begin{figure}[h]
    \centering 
    \includegraphics[width=1\textwidth]{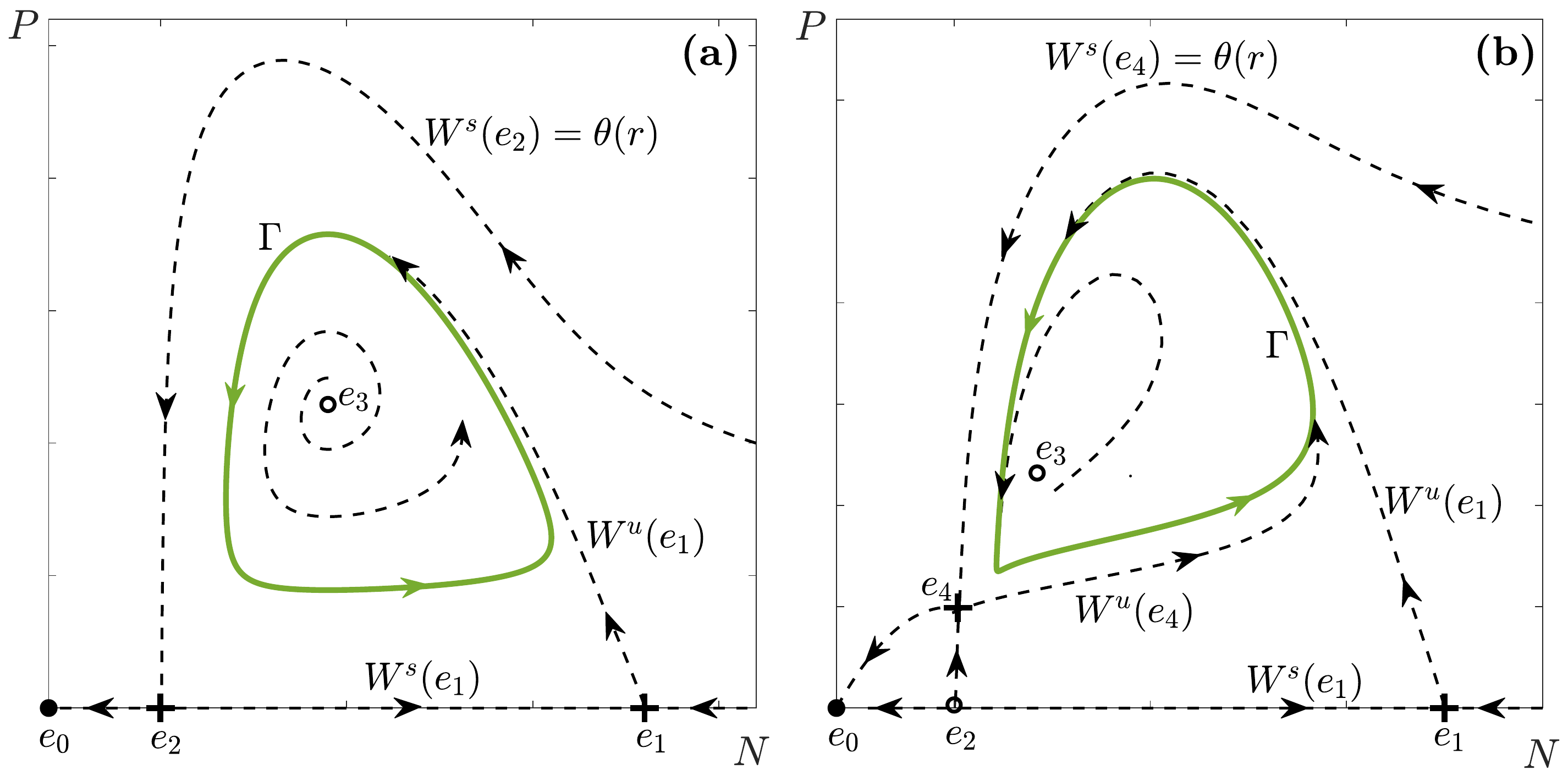}
    \caption{ 
    Schematic phase portraits showing stable (black dots), unstable (black circles) and saddle (black plus signs) equilibria; the stable/unstable manifolds (black dashed curves) of the saddle equilibria; and the (green) limit cycles $\Gamma$ in (a) the autonomous RMA frozen model~\eqref{eq:RM1} with $r\in(1.53,\,2.61)$ and (b) the autonomous May frozen model~\eqref{eq:May1} with $r\in(1.66,\,3.81)$.
    The other parameter values are given in Appendix~\ref{app:maymodel}, Table~\ref{tab:parValues}.
    }
    \label{fig:SchematicPhase}
\end{figure}

The May frozen system can have at most five stationary solutions (equilibria), which are derived by setting $\dot{N} = \dot{P} = 0$ in~\eqref{eq:May1}. In addition to the {\em extinction equilibrium} 
$e_0$, which is always stable,
there is a {\em prey-only} equilibrium $e_1(r)$, the {\em Allee equilibrium}  $e_2$, and
 {\em two coexistence equilibria} $e_3(r)$ and $e_4(r)$, whose stability depends on the system parameters
\begin{equation}\label{eq:May_equilibria}
   e_0=(0,\epsilon/q),~e_1(r) = (r/c,0),~e_2 = (\mu,0),~e_3(r)=\left(N_3(r),P_3(r)\right),~e_4(r)=\left(N_4(r),P_4(r)\right).
\end{equation}
In the above, we include the argument $(r)$ when an equilibrium's position depends on $r$.
The prey population densities of the coexistence equilibria $e_3(r)$ and $e_4(r)$ are the two non-negative roots, denoted $N_3(r)$ and $N_4(r)$ respectively, of the third degree polynomial
\begin{equation}\label{}
N^3 
- \left(\mu -\beta + \frac{r}{c} - \frac{\alpha}{cq} \right)N^2
- \left(\beta\mu + \frac{r(\beta - \mu)}{c} -\frac{\alpha(\nu + \epsilon)}{cq}\right)N
+\left( \frac{r\beta\mu}{c} + \frac{\alpha\nu\epsilon}{cq} \right) = 0,
\end{equation}
and the corresponding predator population densities are given by
$$
P_{i}(r) = \frac{N_{i}(r)+\epsilon}{q},\;\;i=3,4.
$$

The one-parameter bifurcation diagram of the  May frozen system in Fig.~\ref{fig:1parBD}\blueb \, reveals different bifurcations and bistability. 
Most importantly, as $r$ is increased, the coexistence equilibrium 
$e_3(r)$ gives rise to a stable limit cycle $\Gamma(r)$ via a safe supercritical Hopf bifurcation, denoted $H_1$. The cycle exists for a range of $r$, and disappears in a reverse supercritical Hopf bifurcation,  denoted $H_2$, for larger $r$.

\begin{table}[ht]
    \caption{Realistic parameter values for the RMA model~\eqref{eq:RM1} and the May model~\eqref{eq:May1}, estimated from Canada lynx and snowshoe hare data \cite{strohm:2009,tyson:2009}.}
    \label{tab:parValues}
    \begin{center}
        \begin{tabular}{ c|c|c|c}
            Parameter  & Units         & RMA model & May model\\
            \hline
            $r$        &  1/yr                 & $[0,3]$   & $[0,4]$  \\
            $c$        &  ha/(prey$\cdot$yr)   & $0.19$  & $0.22$ \\
            $\alpha$   &  prey/(pred$\cdot$yr) & $800$     & $505$    \\
            $\beta$    &  prey/ha              & $1.5$     & $0.3$    \\ 
            $\chi$     &  pred/prey            & $0.004$   & n/a      \\
            $\delta$   &  1/yr                 & $2.2$     & n/a      \\
            $s$        &  1/yr                & n/a       & $0.85$   \\
            $q$        &  prey/pred            & n/a       & $205$    \\
            $\mu$      &  prey/ha              & $0.03$    & $0.03$   \\
            $\nu$      &  prey/ha              & $0.003$   & $0.003$  \\
            $\epsilon$ &  prey/ha              & n/a       & $0.031$
\end{tabular}
\end{center}
\end{table}

\section{Numerical computations of invariant measures}
\label{App:invariantMeasures}
We estimate the invariant measure $\mu(\varphi_\gamma)$ as the fallowing: 
\begin{enumerate}
    \item We start with a large number $J$ of initial conditions, evenly distributed around the periodic orbit $\Gamma$ and solve the system subject to these initial conditions up to time $T$. This gives $J$ trajectories $x_j(t)$ for $j = 1,2,\ldots,J$ and $t\in[0,T]$.
    \item We consider the final points of all of these trajectories, $x_j(T)$ and compute the phase of cycle for these points $\varphi_{x_j}$, for $j = 1,2,\ldots, J$.
    \item For any point $\gamma \in \Gamma$, suppose that for some $\varepsilon > 0$ there are $K$ points with the respective phases $\varphi_{x_k} \in [\varphi_{\gamma} -\varepsilon,\varphi_{\gamma} + \varepsilon]$, for $k = 1,2, \ldots K$.  We then define the invariant measure $\mu(\varphi_\lambda)$ as:
    $$
    \mu(\varphi_\lambda) = \frac{K}{J}.
    $$
\end{enumerate}

In Fig.~\ref{fig:cli_dist_RM_past_h}\bluee \ we choose $J = 10000$, $T = 100$, and $\varepsilon = 0.1$.

\section{Classical autonomous bifurcation analysis.}
\label{app:cba}

We start with the autonomous RMA frozen model~\eqref{eq:RM1}, consider the climatic parameter $r$ together with the predator mortality rate $\delta$, and examine the bifurcation structure in the $(r,\delta)$ parameter space in Fig.~\ref{fig:2parBD}\bluea. The dynamics are organised by the codimension-two  {\em double-transcritical} bifurcation point $TT$, due to an intersection of two transcritical bifurcation curves, namely $T_1$, along which $e_1(r)$ and $e_2(r)$ meet and exchange stability, and $T_2$, along which $e_1(r)$ and $e_3(r)$ meet and exchange stability. (Since a Hopf bifurcation for a complex variable $z = r\,e^{i\theta}$ is a transcritical bifurcation for the ``amplitude" variable $\rho = r^2$, we expected the unfolding of $TT$ to be the same as one of the unfoldings in the ``amplitude equations" for the Hopf-Hopf bifurcation. This, however, is not the case. The unfolding of $TT$ is akin, although not identical, to the unfolding of the ``amplitude equations" for the Hopf-Hopf bifurcation point in subregion 6 of the ``difficult" case from Ref.~\cite[Sec.8.6]{kuznetsov2013}.)
$TT$ is the origin of the  Hopf $H$ and heteroclinic $h$ bifurcation curves, both of which are subcritical (dashed) near $TT$.
Furthermore, $H$ changes from subcritical (dashed) to supercritical (solid) at the codimension-two {\em generalised Hopf} bifurcation point $GH$, from which the curve $F_l$ of the fold of limit cycles emerges. The stable limit cycle $\Gamma(r)$ shrinks onto $e_3(r)$ along the supercritical (solid) part of $H$, or collides with an unstable limit cycle and disappears along $F_l$.
Then, $F_l$ has another endpoint on $h$. This point is the codimension-two {\em resonant heteroclinic} bifurcation point $Rh$, where $h$ changes from subcritical (dashed) to supercritical (solid). The stable limit cycle $\Gamma(r)$ collides simultaneously with two saddles, $e_1(r)$ and $e_2$, and disappears along the supercritical (solid) part of $h$. Our main focus is on the (green) region of bistability between  oscillatory coexistence $\Gamma(r)$ and extinction $e_0$.
This region is bounded by the three bifurcation curves along which the stable limit cycle $\Gamma(r)$ disappears: the fold of limit cycles $F_l$, the (solid) supercritical part of the Hopf curve $H$, and the (solid) supercritical part of {  the heteroclinic curve} $h$. 
Finally, note that there is a third transcritical bifurcation curve corresponding to $T_0$ in the inset of Fig.~\ref{fig:1parBD}\bluea. This curve is not shown in Fig~\ref{fig:2parBD}\bluea~for clarity reasons; it lies very close to $T_1$ and is not relevant to our study.

For the autonomous May frozen model~\eqref{eq:May1}, we consider the climatic parameter $r$ together with $q$. Here, $q$ specifies the minimum prey-to-predator biomass ratio required for predator population growth, and can be thought of as an `equivalent' of the predator mortality rate from the RMA frozen model~\eqref{eq:RM1}. The qualitative picture, shown in Fig.~\ref{fig:2parBD}\blueb, is very similar to that for the RMA frozen model in Fig.~\ref{fig:2parBD}\bluea. The main difference is that the organising centre for the dynamics is the codimension-two {\em Bogdanov-Takens} bifurcation point $BT$.  Furthermore, instead of the three transcritical bifurcation curves there is just one, denoted $T$, along which $e_1(r)$ and $e_2$ meet and become degenerate, together with a single (dark blue) curve $F_e$ of fold of equilibria, where $e_3(r)$ and $e_4(r)$ become degenerate and disappear.
As a result, the region of ``Extinction or Prey Only" is gone, leaving just three ecologically relevant parameter regions. The heteroclinic bifurcation curve $h$ is replaced by a {\em homoclinic} bifurcation curve $h^*$, along which $\Gamma(r)$ collides with one saddle, namely $e_4(r)$, and disappears. The resonant heteroclinic point $Rh$ is replaced by a {\em resonant homoclinic} point $Rh^*$.
Most interestingly, except for the change from $h$ to $h^*$, the boundary of the (green) region of  bistability between  oscillatory coexistence $\Gamma(r)$ and extinction $e_0$ consists of the same bifurcation curves as in the RMA frozen model.

\end{document}